\documentclass[12pt, reqno]{amsart}
\pdfoutput=1 % force pdflatex engine on arXiv

\usepackage[utf8]{inputenc}
\usepackage[english]{babel}

\usepackage{ifthen}
\usepackage{microtype}
\usepackage[dvipsnames]{xcolor}
\usepackage{comment}
\usepackage{blkarray}
\usepackage{fullpage}
\usepackage{mathtools}
\usepackage{xfrac}
\usepackage{centernot}
\usepackage{enumitem}
\usepackage{mathdots}
\usepackage{multirow}
\usepackage{amssymb}

\usepackage[noadjust]{cite}

\newlist{paraenum}{enumerate}{1}
\setlist[paraenum]{wide, label=(\arabic*)}

\makeatletter
\patchcmd{\@setaddresses}{\indent}{\noindent}{}{}
\patchcmd{\@setaddresses}{\indent}{\noindent}{}{}
\patchcmd{\@setaddresses}{\indent}{\noindent}{}{}
\patchcmd{\@setaddresses}{\indent}{\noindent}{}{}
\makeatother

\usepackage{caption}
\usepackage{subcaption}
\usepackage{tikz}    
\usetikzlibrary{cd}
\usetikzlibrary{backgrounds}
\usetikzlibrary{shapes}
\usetikzlibrary{arrows}
\usetikzlibrary{positioning}
\usetikzlibrary{calc}
\usetikzlibrary{quotes}
\usetikzlibrary{fit}
\usetikzlibrary{decorations.pathreplacing}

\usepackage[pdftex]{hyperref}
\hypersetup{
    colorlinks=true,
    linkcolor={magenta},
    citecolor={blue},
    urlcolor={blue}
}
\usepackage[nameinlink,capitalise]{cleveref}

\numberwithin{equation}{section}
\theoremstyle{plain}
\newtheorem{theorem}{Theorem}[section]

\newtheorem{corollary}[theorem]{Corollary}
\newtheorem{lemma}[theorem]{Lemma}
\newtheorem{proposition}[theorem]{Proposition}

\theoremstyle{definition}
\newtheorem{remark}[theorem]{Remark}
\newtheorem{examplex}[theorem]{Example}
\newenvironment{example}
  {\pushQED{\qed}\examplex}
  {\popQED\endexamplex}

\newtheorem{definition}[theorem]{Definition}

\newtheorem{convention}[theorem]{Convention}
\newtheorem*{convention*}{Convention}

\newcommand{\rr}{\mathbb{R}}

\newcommand{\nn}{\mathbb{N}}

\newcommand{\cm}{\mathcal{M}}
\newcommand{\cn}{\mathcal{N}}
\newcommand{\ct}{\mathcal{T}}
\newcommand{\cg}{\mathcal{G}}

\newcommand{\cz}{\mathcal{Z}}

\DeclareMathOperator{\Pa}{Pa}
\DeclareMathOperator{\pa}{pa}

\DeclareMathOperator{\An}{An}

\DeclareMathOperator{\De}{De}

\DeclareMathOperator{\ch}{ch}
\DeclareMathOperator{\Ch}{Ch}

\newcommand{\indep}{\mathrel{\text{$\perp\mkern-10mu\perp$}}}

\newcommand{\nrnodes}{n}
\newcommand{\edge}{\mathrel{\text{---}}}
\newcommand{\pd}{\mathrm{PD}}
\newcommand{\ot}{\leftarrow}
\newcommand{\Id}{I}

\newcommand{\co}{\mathsf{c}}
\newcommand{\defas}{\coloneqq}
\newcommand{\asdef}{\eqqcolon}

\DeclareMathOperator{\sgn}{sgn}
\newcommand{\hatg}{\hat{\cg}}
\DeclareMathOperator{\low}{low}
\DeclareMathOperator{\high}{high}
\DeclareMathOperator{\adj}{adj}
\newcommand{\hsig}{\hat{\sigma}}
\newcommand\tp[2][-6]{{#2}^{\mkern#1mu T}} % transpose matrix (-6 basic dist, -9 closer)

\definecolor{benpurple}{RGB}{180, 0, 240}

\usepackage{etoolbox}
\makeatletter
\patchcmd{\@maketitle}{\@setauthors}{%
\centering%
\begin{minipage}[b]{.7\linewidth}%
\noindent\@setauthors%
\end{minipage}}
{}{}
\makeatother

\author{Tobias Boege}
\address{Department of Mathematics, KTH Royal Institute of Technology, %, SE-100 44 Stockholm,
Sweden}
\email{post@taboege.de}

\author{Mathias Drton}
\address{TUM School of Computation, Information and Technology, Technical University of Munich, and Munich Center for Machine Learning, Germany}
\email{mathias.drton@tum.de}

\author{Benjamin Hollering}
\address{TUM School of Computation, Information and Technology, Technical University of Munich}
\email{benjamin.hollering@tum.de}

\author{Sarah Lumpp}
\address{TUM School of Computation, Information and Technology, Technical University of Munich}
\email{sarah.lumpp@tum.de}

\author{Pratik Misra}
\address{TUM School of Computation, Information and Technology, Technical University of Munich}
\email{pratik.misra@tum.de}

\author{Daniela Schkoda}
\address{TUM School of Computation, Information and Technology, Technical University of Munich}
\email{daniela.schkoda@tum.de}

\title{Conditional Independence in Stationary Diffusions}


\begin{document}

% {\color{red}
% \section*{Notation conventions}
% \begin{itemize}
%     \item "top node" as opposed to "source node" for the top node of a trek to match our symbol $t$ for the top node
%     \item $\nrnodes$ for the number of nodes in the graph, $V = [\nrnodes]$ throughout the paper.
%     \item $\Id_\nrnodes$ (not just $I$) identity matrix
%     \item $\ell$ only for length of left part of a trek, not as index
%     \item $t$ time or top node of a trek $\rightarrow$ the parameter in the specialization in section 4 has been renamed to $\zeta$, this could also be changed to any other variable which isn't already mentioned here
%     \item $i,j,k,l,m,p,q,s,a, c, ...$ nodes or corresponding matrix indices (could use $u,v,w$ if we need more), in a few cases also used as natural numbers
%     \item $I \indep J \mid K$ whenever possible as opposed to $X_I \indep X_J \mid X_K$
%     \item $u \leftrightarrow v$ for bidirected edges
%     \item The zero pattern of the drift matrix $M$ is defined as follows: $m_{ji} = 0$ if $i \to j \notin E$ (i.e. $m_{ji}$ corresponds to $i \to j$).
%     \item Integration is written $\mathrm{d}x$ not $dx$
% \end{itemize}
% }
% \clearpage
\begin{abstract}
Stationary distributions of multivariate diffusion processes have
recently been proposed as probabilistic models of causal systems in statistics and machine learning. Motivated by these developments, we study stationary multivariate diffusion processes with a sparsely structured drift. Our main result gives a characterization of the conditional independence relations that hold in a stationary distribution. The  result draws on a graphical representation of the drift structure and pertains to conditional independence relations that hold generally as a consequence of the drift's sparsity pattern.
\end{abstract}

\maketitle

% TODO:
% Finalize introduction
% Revise comments/ notation changes (especially section 4 and appendix)
% Check coherent notation

\setcounter{section}{0}

\section{Introduction} 

% Define $P^X$ as stationary solution of SDE
% We write $\{\mathbb{X}(t)\}_{t\geq 0}$ for stochastic processes.

Consider an $\mathbb{R}^\nrnodes$-valued stationary stochastic process $\{\mathbb{X}(t)\}_{t\geq 0}$ that arises as the solution 
of a stochastic differential equation (SDE)
\begin{align}\label{def:stochastic differential equation}
\:\mathrm{d}\mathbb{X}(t)= f(\mathbb{X}(t))\,\mathrm{d}t + D \:\mathrm{d}\mathbb{W}(t),
\end{align}
 understood in the Itô sense \cite{bhattacharya2023continuous}.
Here,  $f: \mathbb{R}^\nrnodes \to \mathbb{R}^\nrnodes$ is a Lipschitz-continuous map that specifies the drift, $D\in \mathbb{R}^{\nrnodes \times \nrnodes}$ is a diagonal and invertible matrix, and $\mathbb{W}$ is a standard $\nrnodes$-dimensional Brownian motion.  Assuming its unique existence, the stationary (marginal) distribution of the process $\mathbb{X}$ depends on $D$ only through the \emph{diffusion matrix} $C =  DD^T$, and we denote the distribution by $P_{f,C}$.  Our interest lies in examining the conditional independence (CI) relations that hold in $P_{f,C}$ when the drift is sparse, i.e., when each component function $f_i$ only depends on some of its input coordinates.
%$x_j$, $j\in[\nrnodes]:=\{1,\dots,\nrnodes\}$. 
Let $X$ be an $\mathbb{R}^\nrnodes$-valued random vector  distributed according to $P_{f,C}$.  Let $I,J,K\subseteq[\nrnodes]:=\{1,\dots,\nrnodes\}$ be disjoint index sets, and denote subvectors of $X$ as $X_I=(X_i)_{i\in I}$.  Then we ask which sparsity restrictions on $f$ imply that $X_I$ and $X_J$ are conditionally independent given $X_K$; in symbols, $X_I\indep X_J\mid X_K$ or briefer $I\indep J\mid K$.  The main results of this paper answer this question via a characterization based on a directed graph that encodes the sparsity assumptions about the map $f$.

% We consider the random vector $X \in \mathbb{R}^\nrnodes$ to arise as the stationary solution of a stochastic differential equation (SDE) understood in the Itô sense
% \begin{align}\label{def:stochastic differential equation}
% \:\mathrm{d}\mathbb{X}(t)= f(\mathbb{X}(t))\,\:\mathrm{d}t + D \:\mathrm{d}\mathbb{W}(t),
% \end{align} with a drift term $f: \mathbb{R}^\nrnodes \to \mathbb{R}^\nrnodes$ whose components $f_i$ are locally $m$-integrable for some $m>\nrnodes$, %$f_i \in L^m_{\text{loc}}(\mathbb{R}^\nrnodes)$
% an invertible diagonal diffusion matrix $D \in \mathbb{R}^{\nrnodes \times \nrnodes}$ and a multivariate Wiener-process $\mathbb{W}$. Such solutions $X$ are also referred to as \textit{stationary diffusion.}
% % State our main question:  what is conditional independence structure in $P^X$
% Our main interest lies in examining the conditional independence (CI) structure in $P^X$, when the function $f$ defining the drift term exhibits sparse structure in the sense that each component $f_i$ only depends on some of the variables $x_j$. 

Before proceeding, we firm up the framework.
%, ensuring unique existence of the contemplated process. 
If $f$ is Lipschitz-continuous, the SDE in \eqref{def:stochastic differential equation}, initialized at a random vector $\mathbb{X}(0)$ that is independent of the driving Brownian motion, has a unique but not necessarily stationary solution.
% with an initial condition $X(0) \sim X_0$ with $x_0$ independent of the Brownian motion always has a unique solution. 
%However, this solution is not necessarily stationary. 
%To ensure existence and uniqueness (even without specifying an initial condition) of a stationary solution, we 
Existence and uniqueness of a stationary solution may be guaranteed in the framework provided by \cite{Huang2015FokkerPlanck}.  To this end, we assume that $f$ has component functions $f_i$ that are locally $m$-integrable for $m>\nrnodes$, and that $f$ satisfies a Lyapunov condition. Referring to the entries of the diagonal diffusion matrix $C =  (c_{ij})$, define the  adjoint Fokker--Planck operator 
\begin{align*}
    \mathcal{L} = \sum_{i=1}^\nrnodes \frac{1}{2}c_{ii}\partial_i^2 + f_i \partial_i.
\end{align*}
 The Lyapunov condition then requires existence of a non-negative function $U \in C^2(\mathbb{R}^\nrnodes)$, the Lyapunov function, fulfilling
 \begin{align}
 \label{eq:lyapunov}
     \lim_{\|x\| \to \infty} U(x) = \infty, \text{ and } \limsup_{\|x\| \to \infty} \mathcal{L}U(x) < 0.
 \end{align}

% \begin{definition}
%     \label{def:valid}
%      We call a drift map $f:\mathbb{R}^\nrnodes\to\mathbb{R}^\nrnodes$ valid if it is Lipschitz-continuous and satisfies the above specified local integrability and Lyapunov conditions.  Via the stationary (marginal) distribution of a solution $\mathbb{X}$ to~\eqref{def:stochastic differential equation}, 
%      each pair $(f,C)$ of a valid drift map $f$ and a positive definite diagonal diffusion matrix $C$ defines a probability  distribution on 
%      $\mathbb{R}^\nrnodes$.  We denote this stationary distribution by $P_{f,C}$.
% \end{definition}
 
In the sequel, we use graphs to represent the sparsity structure of a drift map $f$.
Specifically, we consider directed graphs $\cg = (V, E)$ whose vertex set $V= [\nrnodes]$ indexes the coordinate processes of $\mathbb{X}$.
For all graphs considered in the paper, we  \emph{tacitly assume that the edge set includes all self-loops $i\to i$}.
%Throughout the paper we assume the vertices to be enumerated as $V = [\nrnodes]$. Denoting  the parents of a node $i \in V$ by $\pa(i) = \{j \in[\nrnodes]: (j,i) \in E\}$,
Let $\Pa_\cg(i) = \{j \in[\nrnodes]: (j,i) \in E\}$ denote the set of parents of node $i\in V$ in graph $\cg$.  By our convention on self-loops, $i\in\Pa_\cg(i)$.

\begin{definition}\label{def:graphical_dynamical_model}
We call a drift map $f:\mathbb{R}^\nrnodes\to\mathbb{R}^\nrnodes$ \emph{$\cg$-compatible}  if it is Lipschitz-continuous, satisfies the above local integrability and Lyapunov conditions, and has its sparsity pattern given by $\cg$, that is, each coordinate function $f_i:\mathbb{R}^\nrnodes\to\mathbb{R}$ only depends the values $x_{\Pa_\cg(i)}$ of input coordinates indexed by $\Pa_\cg(i)$, $i\in[\nrnodes]$.
\end{definition}

% We now introduce our main theorem.  If there is a directed path from node $i$ to node $j$ in the directed graph $\cg$, then we say that $i$ is an \emph{ancestor} of $j$, and $j$ is a \emph{descendant} of $i$.  We write $\text{an}_\cg(i)$ and $\text{de}_\cg(i)$ for the set of all ancestors and all descendants of $i$, respectively.  Note that $i\in\text{an}_\cg(i)$ and $i\in\text{de}_\cg(i)$.  

We now introduce our main theorem.  If there is a directed path from node $i$ to node $j$ in the directed graph $\cg$, then $i$ is an \emph{ancestor} of $j$.  Write $\An_\cg(i)$ for the set of all ancestors of $i$.  Note that $i\in\An_\cg(i)$.  
Given $\cg$, we construct an auxiliary graph $\hatg=(V,\hat{E})$ that we call the \emph{trek graph}.  It is obtained by including a \emph{bidirected edge} $i\leftrightarrow j$ in $\hat{E}$ if $\An_\cg(i)\cap\An_\cg(j)\not=\emptyset$. 

\begin{theorem}
    \label{thm:main}
    Let $\cg=(V,E)$ be a directed graph on $V=[\nrnodes]$, and let $\hatg$ be its trek graph.  Let $I,J,K\subseteq[\nrnodes]$ be three disjoint index sets.  Then the following two statements are equivalent: 
    \begin{enumerate}
        \item[(i)] The set $I$ is separated from $J$ by $V \setminus (I \cup J \cup K)$ in the trek graph $\hatg$.
        \smallskip
        \item[(ii)] The conditional independence $X_I \indep X_J \mid X_K$ holds in any random vector $X$ that is distributed according to a stationary distribution $P_{f,C}$ given by a $\cg$-compatible drift map $f:\mathbb{R}^\nrnodes\to\mathbb{R}^\nrnodes$ and a positive definite diagonal matrix $C\in\mathbb{R}^{\nrnodes\times\nrnodes}$. 
    \end{enumerate}
\end{theorem}

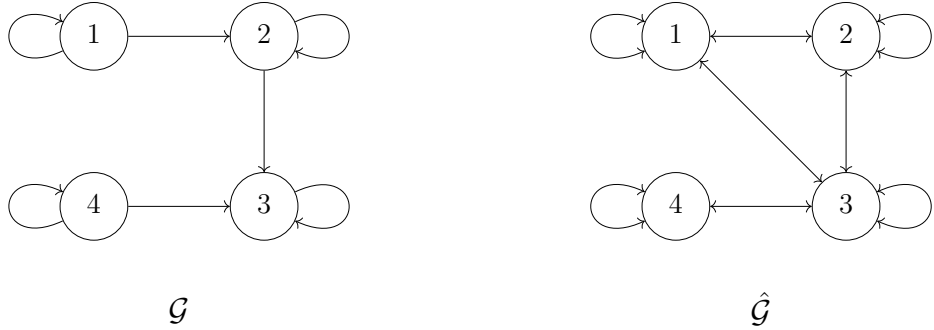
\begin{figure} 
\newcommand{\scalefactor}{0.9} 
\centering
\begin{subfigure}[t]{0.43\textwidth}
    \centering
    \scalebox{\scalefactor}{%
    \begin{tikzpicture}[minimum size = 1cm, every loop/.style={looseness = 7}]
    \node [draw, circle] (1) {\strut$1$};
    \node [right = 1.5 of 1, draw, circle] (2) {\strut$2$};
    \node [below = 1.5 of 2, draw, circle] (3) {\strut$3$};
    \node [left = 1.5 of 3, draw, circle] (4) {\strut$4$};
    \node [below = 0.5cm of $(4.south)!0.5!(3.south)$, scale = 1/\scalefactor] (G) {\strut$\cg$};
    
    \draw [->] (1) -- (2);
    \draw [->] (2) -- (3);
    \draw [->] (4) -- (3);
    \draw[->] (1) edge[out=205,in=155, loop]node {} (1); 
    \draw[->] (2) edge[out=25,in=335, loop]node {} (2); 
    \draw[->] (3) edge[out=25,in=335, loop]node {} (3); 
    \draw[->] (4) edge[out=205,in=155, loop]node {} (4); 
    \end{tikzpicture}}
    \caption{Directed graph on four nodes including all self-loops.}
    \label{fig:example_intro:a}
\end{subfigure}%
\hspace{0.5cm}
\begin{subfigure}[t]{0.43\textwidth}
    \centering
    \scalebox{\scalefactor}{%
    \begin{tikzpicture}[minimum size = 1cm, every loop/.style={looseness = 7}]
    \node [draw, circle] (1) {\strut$1$};
    \node [right = 1.5 of 1, draw, circle] (2) {\strut$2$};
    \node [below = 1.5 of 2, draw, circle] (3) {\strut$3$};
    \node [left = 1.5 of 3, draw, circle] (4) {\strut$4$};
    \node [below = 0.5cm of $(4.south)!0.5!(3.south)$, scale = 1/\scalefactor] (Ghat) {\strut$\hatg$};
    
    \draw [<->] (1) -- (2);
    \draw [<->] (2) -- (3);
    \draw [<->] (1) -- (3);
    \draw [<->] (3) -- (4);
    \draw[<->] (1) edge[out=205,in=155, loop]node {} (1); 
    \draw[<->] (2) edge[out=25,in=335, loop]node {} (2); 
    \draw[<->] (3) edge[out=25,in=335, loop]node {} (3); 
    \draw[<->] (4) edge[out=205,in=155, loop]node {} (4);
    \end{tikzpicture}}
    \caption{Corresponding bidirected trek graph.}
    \label{fig:example_intro:b}
\end{subfigure}
\caption{Directed graph $\cg$ and its trek graph $\hatg$.}
\label{fig:example_intro}
\end{figure}

\begin{example}
\label{ex:trekgraph}
   Consider the graph $\cg$ depicted in \Cref{fig:example_intro:a}, and let $X$ be a random vector following a stationary distribution  given by a $\cg$-compatible drift map and a positive definite diagonal diffusion matrix. The graph $\cg$ has the ancestor sets,
   \begin{alignat*}{4}
   \An_\cg(1)&=\{1\}, &\quad\An_\cg(2)&=\{1,2\}, & \quad\An_\cg(3)&=\{1,2,3,4\}, &\quad\An_\cg(4)&=\{4\}.
   \end{alignat*}
   As only the pairs of nodes $(1,4)$ and $(2,4)$ do not have a common ancestor, the trek graph $\hat{\cg}$ is missing precisely the edges $1 \leftrightarrow 4$ and $2 \leftrightarrow 4$; see \Cref{fig:example_intro:b}.  For subsets of nodes $I$, $J$, and $S$, write $I\perp_{\hatg} J\mid S$ if $I$ and $J$ are separated by $S$ in $\hatg$ (i.e., every path from a node $i\in I$ to a node $j\in J$ intersects $S$).  Then \Cref{thm:main} implies the following marginal and conditional independencies:
   \begin{alignat*}{5}
       \{1,2\}&\perp_{\hatg}\{4\}\mid \{3\} &\quad\implies&&\quad X_{\{1,2\}} &\indep X_4,\\
       \{1\}&\perp_{\hatg}\{4\}\mid \{2,3\} &\quad\implies&&\quad X_1 &\indep X_4,\\
       \{2\}&\perp_{\hatg}\{4\}\mid \{1,3\} &\quad\implies&&\quad X_2 &\indep X_4,\\
       \{1\}&\perp_{\hatg}\{4\}\mid \{3\} &\quad\implies&&\quad X_1 &\indep X_4\mid X_2,\\
       \{2\}&\perp_{\hatg}\{4\}\mid \{3\} &\quad\implies&&\quad X_2 &\indep X_4\mid X_1.
   \end{alignat*}
   Here, the first listed separation implies the other separations graphically, and the first listed marginal independence implies the other independencies probabilistically by the semi-graphoid axioms of conditional independence \cite[\S2.2]{Studeny}. 

    \Cref{thm:main} further states that no marginal or conditional independencies  other than the ones listed above hold for \emph{all} considered stationary distributions. For instance, \Cref{thm:main} implies existence of a $\cg$-compatible drift map $f$ and a positive definite diagonal diffusion matrix $C$ such that 
    \[
    X_1 \not\indep X_4 \mid X_3
    \]
    for $X \sim P_{f,C}$. We demonstrate how to construct such distributions in \Cref{sec:LyapunovCI}.
\end{example}

The proof of \Cref{thm:main} will be developed in the subsequent sections.  The fact that separation in the trek graph implies conditional independence will be derived with the help of the Fokker--Planck equation that characterizes the density of $P_{f,C}$ (\Cref{sec:MarginalIndep}). In constructing counterexamples that exhibit conditional dependence when separation fails, we strive hard to structure them as simple as possible.  Specifically, our counterexamples take $f$ to be linear, which gives an Ornstein--Uhlenbeck process for which we let the diffusion matrix be a multiple of the identity (\Cref{sec:LyapunovCI}). 

\begin{remark}[Applied motivation]
Stationary diffusion processes have emerged as a new framework for statistical analysis of causal systems.  Indeed, in many applications it is difficult if not impossible to observe a system more than once; e.g., in biology, measurement may destroy the observed collection of cells. %\cite{young2019identifying}. 
The available data then constitute a sample of independent observations, for which the new paradigm retains an implicit temporal perspective by considering one-time snapshots of independent stationary dynamic processes. This approach was pioneered using stationary Ornstein--Uhlenbeck processes, resulting in Gaussian models termed \emph{Lyapunov models} \cite{fitch2019learning, varando2020graphical}.  Research progresses on issues like identifiability of drift parameters, statistical estimation, and new applications \cite{dettling2022identifiability, dettling2022lasso,wang2023dictys,rohbeck2024bicycle}.  Moreover, extensions to non-linear diffusions were proposed in \cite{lorch2024causal}.  \Cref{thm:main} advances this line of research via a characterization of conditional independence  that can be used to disprove parameter identifiability \cite{dettling2022identifiability}. It also reveals that even for general drift maps, the trek graph $\hatg$ may be estimated by  learning a marginal independence graph \cite{marginalIndependence2023,varando2020graphical}.
\end{remark}

\begin{remark}[Relation to structural equation models]
Working with stationary distributions of sparse multivariate diffusion processes has particular appeal when modeling causal systems with feedback loops, which present mathematical and interpretational challenges for the classical structural equation models (SEMs) \cite{Peters2017, drton2018algebraic}. For the latter, the well-known d-separation criterion yields a graphical characterization of conditional independence relations \cite{Handbook}. In~our \Cref{thm:main}, a marginal independence $X_I\indep X_J$ (with conditioning set $K=\emptyset$) arises from the absence of common ancestors of nodes in $I$ and nodes in $J$, which is precisely the same phenomenon as in d-separation. However, as will be evident from our later proof, \Cref{thm:main} shows that for stationary diffusions the marginal independencies are the only source of independence relations, which radically differs from the case of d-separation.

% The Ornstein--Uhlenbeck process is a natural modelling choice in various fields, including biology, finance, and survival analysis. For instance, the process is widely used in both phylogenetics (see e.g.~\cite{rohlfs2014modeling}) and single-cell genomics (see e.g.~\cite{matsumoto2016scoup}) to model the evolution of gene expression levels across species and their dynamics within individual cells. 
\end{remark}

% mention identifiability/ lasso/ general structure learning via MIs also here?

% * a bit briefer on Ornstein--Uhlenbeck
% * citations to motivate (identifiability paper SIAM, bicycle,....)

%%%%%%%%%%%%%%%%%%%%%%%%%%%%%%%%%%%%%%%%%%%%%%%%%%%%%
% Results (adapted from section 4)

% Organization 
The remainder of the paper is organized as follows.
% Section 2
\Cref{sec:MarginalIndep} treats the characterization of marginal independence and describes how to infer the conditional independence relations implied by these marginal independencies by means of a connected set Markov property for the trek graph.
% Section 3
In \Cref{sec:LyapBackground}, we set up notation, formally introduce the Lyapunov model and review some facts on Gaussian distributions. 
% Section 4 (more detailed?)
In \Cref{sec:LyapunovCI}, we prove the main result by showing that all conditional independence relations that hold in the Lyapunov model originate from marginal independencies.  To this end, we reduce the problem to graphs obtained from 1-clique sums of treks. We then use a combination of algebraic and analytic methods to establish the result in this simpler case.

\section{Marginal Independence and Implied Conditional Independencies}
\label{sec:MarginalIndep}

This section proves one direction in \Cref{thm:main}, namely, that $V \setminus (I \cup J \cup K)$ separating $I$ and $J$ in the trek graph $\hatg$ implies the conditional independence $X_I \indep X_J \mid X_K$.  We start with the case $K=\emptyset$, i.e., marginal independence statements. Then, we show that these statements imply all claimed conditional independence statements. Before giving the proofs, we introduce graphical concepts, as used in this as well as later sections of the paper.

\subsection{Graphs} 

For a directed graph $\cg = (V, E)$ with edge set $E \subseteq V \times V$, we write $i \to j$ if $(i,j) \in E$.  We denote the set of nodes or edges of a specific graph $\cg$ by $V(\cg)$ or $E(\cg)$, respectively.  The set $\Pa_\cg(i) = \{j \in V: j \to i\}$ collects the parents of a node $i\in V$, and $\Ch_\cg(i) = \{j \in V: i \to j\}$ is its set of children. 
A~directed path is a sequence of 
%distinct
nodes $(i_1, \dots, i_m)$ such that $i_l \to i_{l+1} \in E$ for $1\le l<m$, and $m-1$ is the length of this path. A~path of length~$0$ consists of a single node.
As defined in the introduction, $\An_\cg(i)$ is the set of ancestors of $i$, i.e., the nodes $j$ from which there is a directed path to~$i$.
This~definition is conjunctively applied to sets, so $\An_\cg(I) = \bigcup_{i \in I}\An_\cg(i)$. The set of descendants of a node~$i$ is denoted $\De_\cg(i)$ and consists of all nodes $j$ such that there is a directed path from $i$ to~$j$. Typically, it is clear which graph $\cg$ is concerned, and we  drop the subscript, writing $\An(i)\equiv \An_\cg(i)$ and, similarly,  $\Pa(i)$ and $\Ch(i)$. The respective sets where the node itself is excluded are denoted in lowercase, e.g., $\pa(i)=\Pa(i)\setminus\{i\}$ and $\ch(i)=\Ch(i)\setminus\{i\}$.

Combining two directed paths yields a trek: An $(\ell,r)$-trek between nodes $i_\ell$ and $j_r$ is a sequence of 
%distinct
nodes $(i_\ell, \dots, i_1, t, j_1, \dots, j_r)$ such that both $(t, i_1, \dots, i_\ell)$ and $(t, j_1, \dots, j_r)$ are directed paths.  Node $t$ is the top node of the trek.  If there is a trek between two nodes $i$ and $j$, then the top node is a common element of $\An(i)$ and $\An(j)$.  Thus, the trek graph $\hatg$ from \Cref{thm:main} has two nodes adjacent precisely if there is a trek between them.

A subset of nodes $I \subseteq V$ is ancestral if $\An(I) = I$. A trivial example of an ancestral set is a source node, meaning a node without parents.  
If a directed path starts and ends at the same node, then it forms a directed cycle. If the graph $\cg$ does not contain any directed cycles, then it is termed a directed acyclic graph (DAG).  The vertex set of a DAG may be ordered topologically as $(i_1, \dots, i_n)$ such that for $l \neq k$, $i_l\in\An(i_k)$ implies that $l<k$.

% $i_l < i_{l+1}$
% We write $i < j$ if $i$ is a non-descendant of $j$. If $\cg$ is a DAG, it exhibits a topological ordering $(i_1, \dots, i_n)$ such that $i_l < i_{l+1}$. 

% A directed cycle is a directed path with $i_1 = i_m$. A graph $\cg$ without directed cycles is called directed acyclic graph (DAG).  
% We write $i < j$ if $i$ is a non-descendant of $j$. If $\cg$ is a DAG, it exhibits a topological ordering $(i_1, \dots, i_n)$ such that $i_l < i_{l+1}$. 

\subsection{Marginal Independence}

The first result treats the implication $(i)\Rightarrow(ii)$ in \Cref{thm:main} when the conditioning set $K$ is empty, in which case the independence is marginal.

\begin{theorem} \label{thm:MarginalIndependence}
Let $\mathcal{G}=(V,E)$ be a directed graph on $V=[\nrnodes]$. If $f$ is $\mathcal{G}$-compatible, then $X \sim P_{f, C}$ satisfies $X_I \indep X_J$ for all sets of nodes $I, J \subseteq V$ without trek between them.
\end{theorem}
This theorem generalizes an observation concerning Ornstein--Uhlenbeck processes made in \cite{varando2020graphical}, and we prove it using the Fokker--Planck equation.  To this end, we draw on the next lemma, which follows from combining Theorem A and Proposition 2.1 in \cite{Huang2015FokkerPlanck}.  Here, $C_0^\infty(\mathbb{R}^{\nrnodes})$ is the set of smooth functions with compact support.

\begin{lemma}\label{lem:Fokker-Planck} If the Lyapunov condition is fulfilled, then \eqref{def:stochastic differential equation} has a unique stationary solution admitting a density $p$. This density is the unique density fulfilling the Fokker--Planck equation
  \begin{align*}
    \int_{\mathbb{R}^{\nrnodes}}{}\mathcal{L}\phi(x)p(x)\:\mathrm{d}x = 0 \text{ for all } \phi \in C_0^\infty(\mathbb{R}^{\nrnodes}) .
    \end{align*} 
\end{lemma} 

% The lemma follows from combining Theorem A and Proposition 2.1 in \cite{Huang2015FokkerPlanck}. 

\begin{proof}[Proof of \Cref{thm:MarginalIndependence}] 
By \Cref{lem:Fokker-Planck}, the stationary distribution $P_{f, C}$ has a density $p$.
We begin by showing that for any ancestral set $A\subseteq V$, the marginal density $p_A$ of $X_A$ satisfies the Fokker--Planck equation for the SDE restricted to $A$. Let $A^\co = V \setminus A$ be the set of all remaining nodes. Fix a function $\phi \in C_0^\infty(\mathbb{R}^{\nrnodes})$ with $\phi(x)$ only depending on $x_A$. Then the Fokker--Planck equation for the overall density $p$ yields that
\begin{multline*}
 0 =  \int_{\mathbb{R}^{\nrnodes}}{}\mathcal{L}\phi(x)p(x)\:\mathrm{d}x 
 =\int_{\mathbb{R}^{\nrnodes}}{}\left( \sum_{i=1}^\nrnodes \frac{1}{2}c_{ii}\partial_i^2 \phi(x) + f_i(x) \partial_i \phi(x) \right) p(x)\:\mathrm{d}x \\
 = \int_{\mathbb{R}^{\nrnodes}}{}\left( \sum_{i\in A} \frac{1}{2}c_{ii}\partial_i^2 \phi(x) + f_i(x) \partial_i \phi(x) \right) p(x)\:\mathrm{d}x 
    \end{multline*} 
because $\partial_i \phi(x) \equiv 0$ for $i \notin A$ as $\phi$ only depends on $x_A$. Now, the ancestrality of $A$ implies that $f_i$ for $i\in A$ only depends on $x_A$.  Rewriting the above integral, we obtain, by Fubini,
\begin{align*}
 0 
 %= 
 % &\int_{\mathbb{R}^\nrnodes}{}\left( \sum_{i\in A} \frac{1}{2}c_{ii}\partial_i^2 \phi(x) + f_i(x) \partial_i \phi(x) \right) p(x)\:\mathrm{d}x  \\
  = &\int_{\mathbb{R}^A}{}\int_{\mathbb{R}^{A^{\co}}}{}\left( \sum_{i\in A} \frac{1}{2}c_{ii}\partial_i^2 \phi(x) + f_i(x) \partial_i \phi(x) \right) p(x)\:\mathrm{d}x_{A^{\co}} \:\mathrm{d}x_A \\
  = &\int_{\mathbb{R}^A}{}\left( \sum_{i\in A} \frac{1}{2}c_{ii}\partial_i^2 \phi(x) + f_i(x) \partial_i \phi(x) \right) \int\limits_{\mathbb{R}^{A^{\co}}}{}p(x)\:\mathrm{d}x_{A^{\co}} \:\mathrm{d}x_A \\
  = &\int_{\mathbb{R}^A}{}\left( \sum_{i\in A} \frac{1}{2}c_{ii}\partial_i^2 \phi(x_A) + f_i(x_A) \partial_i \phi(x_A) \right) p_A(x_A) \:\mathrm{d}x_A,
    \end{align*} 
  which is the Fokker--Planck equation for a restricted SDE of the form
  \begin{align*}
    \:\mathrm{d}\mathbb{X}_{A}(t) = f_A(\mathbb{X}_{A}(t)) \:\mathrm{d}t + D_{A, A} \:\mathrm{d}\mathbb{W}(t)
    \end{align*}
    with $D_{A,A}D_{A,A}^T=C_{A,A}$.
Applying this result to $A=\An(I) \cup \An(J)$ gives a Fokker--Planck equation for the marginal density of  $X_{\An(I) \cup \An(J)}$. To show the independence statement, we show that the product of marginal densities $p_{\An(I)} p_{\An(J)}$ fulfills the same Fokker--Planck equation. Let $\tilde{I} = \An(I)$, $\tilde{J} = \An(J)$, and $\phi \in C_0^\infty(\mathbb{R}^A)$. Since there is no trek between $I$ and $J$, the sets $\tilde{I}$ and $\tilde{J}$ are disjoint and partition $A=\tilde{I}\cup\tilde{J}$. Therefore,
\begin{equation}
\begin{aligned}
&\int\limits_{\mathbb{R}^A}{}\left( \sum_{i\in A} \frac{1}{2}c_{ii}\partial_i^2 \phi(x_A) + f_i(x_A) \partial_i \phi(x_A) \right)(p_{\tilde{I}} p_{\tilde{J}})(x_A)\:\mathrm{d}x_A \\ =&\int\limits_{\mathbb{R}^A}{}\left( \sum_{i\in \tilde{I}} \frac{1}{2}c_{ii}\partial_i^2 \phi(x_A) + f_i(x_A) \partial_i \phi(x_A) \right)(p_{\tilde{I}} p_{\tilde{J}})(x_A)\:\mathrm{d}x_A \\ &+\int\limits_{\mathbb{R}^A}{}\left( \sum_{i\in \tilde{J}} \frac{1}{2}c_{ii}\partial_i^2 \phi(x_A) + f_i(x_A) \partial_i \phi(x_A) \right)(p_{\tilde{I}} p_{\tilde{J}})(x_A)\:\mathrm{d}x_A. \label{eq:Fokker-Planck union}
\end{aligned}
\end{equation}
We rewrite the first integral as
\begin{align*}
&\int\limits_{\mathbb{R}^A}{}\left( \sum_{i\in \tilde{I}} \frac{1}{2}c_{ii}\partial_i^2 \phi(x_A) + f_i(x_A) \partial_i \phi(x_A) \right)(p_{\tilde{I}} p_{\tilde{J}})(x_A)\:\mathrm{d}x_A \\ =&\int\limits_{\mathbb{R}^{\tilde{J}}}{}\int\limits_{\mathbb{R}^{\tilde{I}}}{}\left( \sum_{i\in \tilde{I}} \frac{1}{2}c_{ii}\partial_i^2 \phi(x_A) + f_i(x_A) \partial_i \phi(x_A) \right)(p_{\tilde{I}} p_{\tilde{J}})(x_A)\:\mathrm{d}x_{\tilde{I}} \:\mathrm{d}x_{\tilde{J}} \\ 
 =&\int\limits_{\mathbb{R}^{\tilde{J}}}{}\int\limits_{\mathbb{R}^{\tilde{I}}}{}\left( \sum_{i\in \tilde{I}} \frac{1}{2}c_{ii}\partial_i^2 \phi(x_A) + f_i(x_A) \partial_i \phi(x_A) \right)p_{\tilde{I}} (x_{\tilde{I}}) \:\mathrm{d}x_{\tilde{I}}  p_{\tilde{J}}(x_{\tilde{J}}) \:\mathrm{d}x_{\tilde{J}}. 
  \end{align*} 
  The Fokker--Planck equation for $\tilde{I}$, with test function $x_{\tilde{I}}\mapsto \phi(x_{\tilde{I}},x_{\tilde{J}})$ for fixed $x_{\tilde{J}}$, yields that the inner integral vanishes. Hence, the overall integral vanishes as well. Reversing the roles of $I$ and $J$ shows that also the second summand in \eqref{eq:Fokker-Planck union} is zero.  In conclusion, $p_{\An(I)} p_{\An(J)}$ fulfills the Fokker--Planck equation for $p_{\An(I)\cup \An(J)}$, i.e., $p_{\An(I)} p_{\An(J)}$ is a density of $X_{\An(I)\cup \An(J)}$.  Consequently, $X_{\An(I)}\indep X_{\An(J)}$ and the decomposition property of conditional independence implies our claim that $X_{I}\indep X_{J}$ because $I\subseteq \An(I)$ and $J\subseteq \An(J)$.
\end{proof}

\subsection{Implied Conditional Independencies}
\label{sec:ImpliedCI}

% \begin{itemize}
%     \item Constructing a bidirected graph (Pratik) (with example)
%     \item talk about global M.P. and state a characterization theorem:
%     \begin{itemize}
%         \item global M.P. gives CIs [Ref e.g. \cite{Drton2008MarginalIndependence}]
%         \item State some computation time ...
   
%     \item Example which shows how paths in the bidirected graph basically correspond to zig-zags in the original graph
%     \end{itemize}
% \end{itemize}

Let $\hatg=(V,\hat{E})$ be the trek graph associated to the given directed graph $\cg$.  By definition, $(i,j)\in\hat{E}$ if and only if there is a trek between nodes $i$ and $j$.  Observe that $(i,j) \in \hat{E}$ if and only if $(j,i) \in \hat{E}$, and, reflecting the structure of treks, we draw the edge bidirected as $i\leftrightarrow j$.  \Cref{thm:MarginalIndependence} clarifies that the trek graph encodes marginal independencies, in particular,
\begin{equation}
\label{eq:pairwiseMP}
i\leftrightarrow j \notin\hat{E} \implies X_i\indep X_j
\end{equation}
in random vectors $X\sim P_{f,C}$ with $f$ a $\cg$-compatible drift map.
A more refined interpretation of the trek graph allows us to derive all conditional independencies that are implied by the marginal statements from \Cref{thm:MarginalIndependence}.   Recall the notation $I\perp_{\hatg} J\mid S$  from \Cref{ex:trekgraph}, which indicates that the subsets of nodes $I$ and $J$ are separated by subset $S$ in $\hatg$.

% In the previous subsection we showed that for any random vector $X$ that is distributed according to a stationary distribution given by a $\cg$-compatible drift, $X_i$ is marginally independent of $X_j$ if there does not exist any trek between them in $\cg$.
% %in the graphical dynamical model $\cd_{\cg}$ if there does not exist any trek between the vertices $i$ and $j$ in $\cg$. 
% This allows us to represent the marginal independence structure of $\cg$ by using a bidirected graph. A bidirected graph $\hatg=(V,E)$ is a graph whose edges satisfy the property that $(i,j) \in E$ if and only if $(j,i) \in E$, and thus the edges are drawn as bidirected $i\leftrightarrow j$. In order to represent the marginal structure of $\cg$, a bidirected edge is drawn between $i$ and $j$ in $\hatg$ if and only if there exists a trek connecting $i$ and $j$ in $\cg$. Such representations have also been used to study the marginal independence structures of linear SEMs in \cite{marginalIndependence2023}.

\begin{theorem}
\label{thm:Section2:CIImplied}
Let $\cg=(V,E)$ be a directed graph on $V=[\nrnodes]$, and consider a random vector $X\sim P_{f,C}$, where $f:\mathbb{R}^\nrnodes\to\mathbb{R}^\nrnodes$ is a $\cg$-compatible drift map and $C\in\mathbb{R}^{\nrnodes\times\nrnodes}$ is positive definite and diagonal.  Let $I,J,K\subseteq[\nrnodes]$ be three disjoint index sets, and let $\hatg$ be the trek graph of $\cg$.   Then 
\[
I\perp_{\hatg} J\mid V \setminus (I \cup J \cup K) \quad\implies\quad
X_I \indep X_J \mid X_K.
\]
%
% Let $\cg$ be a directed graph and $\hatg$ be a bidirected graph such that $i \leftrightarrow j \in \hatg$ if and only if there is a trek from $i$ to $j$ in $\cg$. If $I$ is separated from $J$ by $V \setminus (I \cup J \cup K)$ in $\hatg$, then $I \indep J \mid K$ holds for any random vector distributed according to a stationary distribution given by a $\cg$-compatible drift.
% %in the graphical dynamical model $\cd_\cg$. 
\end{theorem}
\begin{proof}
In the literature on probabilistic graphical models, a random vector $X$ that satisfies \eqref{eq:pairwiseMP} is said to satisfy the pairwise Markov property relative to the bidirected graph $\hatg$ \cite{Drton2008MarginalIndependence}.  However, \Cref{thm:MarginalIndependence} shows that a richer set of marginal independencies hold.  Call a subset $A\subseteq V$ connected in $\hatg$ if any two vertices $i, j \in A$ are connected by a path in $\hatg$ all of whose vertices are in $A$ (i.e., $A$ induces a connected subgraph of $\hatg$).  Let 
$\text{Sp}_{\hatg}(A)$ be the set of all nodes that are in $A$ or are adjacent to an element of $A$ via an edge of $\hatg$.  Then \Cref{thm:MarginalIndependence} implies that $X$ satisfies the connected set Markov property relative to $\hatg$, which requires that 
\[
X_A \indep X_{V \setminus \text{Sp}_{\hatg}(A)} \quad\forall \; A\subset V, \text{ non-empty and connected in } \hatg.
\]
By general properties of conditional independence, the connected set Markov property is equivalent to the global Markov property for bidirected graphs \cite{Drton2008MarginalIndependence}.  The global Markov property states that 
\[
I\perp_{\hatg} J\mid V \setminus (I \cup J \cup K) \quad\implies\quad
X_I \indep X_J \mid X_K,
\]
as was the claim of our theorem.
\end{proof}

We illustrate the theorem in another example.

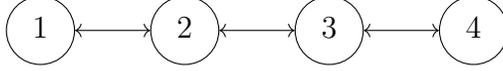
\begin{figure}
\centering
\begin{tikzpicture}
\node [draw, circle] (4) {\strut$2$};
\node [right = 1 of 4, draw, circle, radius = .15cm] (3) {\strut$3$};
\node [right = 1 of 3, draw, circle] (1) {\strut$4$};
\node [left = 1 of 4, draw, circle] (2) {\strut$1$};

\draw [<->] (3) -- (4);
\draw [<->] (3) -- (1);
\draw [<->] (2) -- (4);
\end{tikzpicture}
\caption{A bidirected graph $\hatg$.}
\label{fig:bidirectedGlobalMarkov}
\end{figure}

\begin{example}
Let $\hatg$ be the bidirected graph from \Cref{fig:bidirectedGlobalMarkov}.  The connected set Markov property requires
\[
1\indep \{3,4\},\quad \{1,2\}\indep 4.
\]
As $1$ is separated from $4$ by $3$ in $\hatg$, the global Markov property additionally states $1 \indep 4 \mid 2$. The other conditional (and not marginal) independencies required by the global Markov property are
\[
1\indep 3 \mid 4,\quad 1\indep 4 \mid 3,\quad 2\indep 4 \mid 1.%,\quad 1\indep 4 \mid 2. 
\qedhere
\]
\end{example}

%Although the global Markov property is more exhaustive, it can be shown that a distribution satisfies the global Markov property if and only if it satisfies the connected set Markov property \cite{ancestralMarkovModels2002}. 
In the next sections, we show that \Cref{thm:Section2:CIImplied} already captures all conditional independencies that hold for general $\cg$-compatible drift. To this end, we consider Ornstein--Uhlenbeck processes and the graphical Lyapunov models obtained from their stationary distributions.

\section{Ornstein--Uhlenbeck Processes and Lyapunov Models}
\label{sec:LyapBackground}

This section briefly reviews Ornstein--Uhlenbeck processes and the associated graphical continuous Lyapunov models. We also recall basic facts regarding Gaussian conditional independence that will be used in \Cref{sec:LyapunovCI}.

\subsection{Ornstein--Uhlenbeck Processes}

The \emph{Ornstein--Uhlenbeck process} is the only non-trivial time-homogeneous Markov process that is both Gaussian and stationary \cite{doob1942brownian}.  In the $\nrnodes$-dimensional case, it is defined by the SDE
\begin{equation} \label{eq:ou_multidim}
    \mathrm{d} \mathbb{X}(t) = M (\mathbb{X}(t)-\mu) \,\mathrm{d}t + D \, \mathrm{d} \mathbb{W}(t)
\end{equation}
with an invertible \emph{drift matrix} $M \in \rr^{\nrnodes\times \nrnodes}$ encoding the interactions among the coordinates of $\mathbb{X}(t)$. 
The affine linear drift term introduces a restoring force towards the mean vector $\mu$.
%with strength proportional to the current distance from $\mu$. 
Again we only consider the case with $D$ and the diffusion matrix $C=D D^T$ diagonal.
%, but in general the process can be defined without this assumption.

Evidently, \eqref{eq:ou_multidim} is a special case of
\eqref{def:stochastic differential equation} with $f$ being an affine
function given by $M$ and $\mu$.  In this case, the Lyapunov condition
on $f$ is satisfied if $M$ is a \emph{stable} matrix, that is, the
real parts of the eigenvalues of $M$ are negative.  For $M$ stable,
$\mathbb{X}(t)$ admits a unique stationary distribution that is Gaussian with mean $\mu$ \cite{jacobsen1993brief, risken1989fokker}.  The positive definite covariance matrix of the distribution is the unique solution $\Sigma\equiv\Sigma(M,C)$ of the \emph{continuous Lyapunov equation}
\[
M\Sigma + \Sigma \tp[-0]{M} = -C.
\]
% Since the diffusion matrix $C$ 
% %arising from the diffusion term of the Ornstein--Uhlenbeck process 
% is symmetric, the solution $\Sigma$ of the Lyapunov equation 
% %for a stable matrix $M$ given a positive definite matrix $C$ 
% is also symmetric. 
We note that the solution $\Sigma$ is invariant under joint rescaling of $M$ and $C$. 

Given our interest in conditional independence structure, we restrict
ourselves, without loss of generality, to stationary distributions of
Ornstein--Uhlenbeck processes with mean $\mu = 0$. For a directed
graph $\cg$, the class of Gaussian distributions of interest 
naturally corresponds to the set of covariance matrices $\Sigma$
that can be achieved using stable matrices $M$ for which the linear
map $x \mapsto Mx$ is $\cg$-compatible.  In other words, we consider
stable matrices $M$ whose support (i.e., non-zero entries) is
contained in the edge set $E(\cg)$.

% As hinted earlier, we are interested in the Lyapunov model that arises from restricting to linear $f$ in the definition of the $\cg$-compatible drifts. In other words, we consider equilibrium distributions of the Ornstein--Uhlenbeck process with long-term mean $\mu = 0$.
%, which we now formally define. 
%Given a directed graph $\cg$ with self-loops, we consider the distributions in the graphical dynamical model of $\cg$ that arise from the Ornstein--Uhlenbeck process with long-term mean $\mu = 0$ as introduced in \cref{def:graphical_dynamical_model}. That means, we assume a centered multivariate random vector $X = (X_1,\dots,X_)^T$ to be a single observation of such a process in equilibrium. 
\begin{definition}[\cite{varando2020graphical,dettling2022identifiability}]
    Let $\cg = (V, E)$ be a directed graph on $V = [\nrnodes]$. Let $\rr^E$ be the set of matrices $M = (m_{ij}) \in \rr^{\nrnodes \times \nrnodes}$ such that $m_{ji} = 0$ if $i \to j \notin E$ and denote by $\pd_\nrnodes$ the cone of $\nrnodes \times \nrnodes$ positive definite matrices. Then for any fixed $C \in \pd_\nrnodes$, the (graphical continuous) \emph{Lyapunov model} of $\cg$ is the family of all multivariate Gaussian distributions $\cn(0,\Sigma)$ with covariance matrix $\Sigma$ in the~set
    \[
    \cm_{\cg,C} = \{ \Sigma \in \pd_\nrnodes ~:~ M \Sigma + \Sigma M^T + C = 0 \text{ for a stable matrix } M \in \rr^E\}.
    \]
\end{definition}
In the sequel, we identify the model with its set  of covariance matrices and refer to $\cm_{\cg,C}$ as the Lyapunov model. We denote the entries of a covariance matrix $\Sigma$ by $\sigma_{ij}$. 

\begin{convention}
We denote the $\nrnodes$-dimensional identity matrix by $\Id_\nrnodes$ or equivalently by $\Id_V$ for a set $V$ with $|V| = \nrnodes$. We will typically fix $C = 2 \Id_\nrnodes$ and write $\cm_\cg = \cm_{\cg, 2 \Id_\nrnodes}$. This~assumption guarantees that $\Id_\nrnodes \in \cm_{\cg, C}$ which is a desirable property for any Gaussian model. While the drift matrix $M$ is typically thought of as the weighted adjacency matrix of~$\cg$, we will usually suppress the self-loops in subsequent drawings of graphs for simplicity. 
\end{convention}

\subsection{Conditional Independence in Multivariate Gaussians}

% This section sets up notation and summarizes basic facts about marginals and conditionals of a multivariate Gaussian distribution. In particular, we operate on symmetric matrices 

Consider a symmetric matrix $\Sigma$ with rows and columns indexed by a finite set~$V$. For $I,J\subseteq V$, let $\Sigma_{I,J}$ denote the submatrix with rows indexed by $I$ and columns by~$J$. Submatrices of the form $\Sigma_K \defas \Sigma_{K,K}$ are \emph{principal}. 
% We usually denote index sets (i.e., subsets of $V$) by $I, J, K, L, \ldots$ and individual indices by lowercase $i, j, k, l, \ldots$. 
Especially in subscripts, we will often abbreviate a union $I \cup J$ of index sets to~$IJ$. A single index $i$ is used interchangeably with its singleton set $\{i\}$ in these conventions.
So, $iK\defas \{i\} \cup K$ and $K \setminus i\defas K \setminus \{i\}$.
%$V \setminus ijK\defas V \setminus (\{i,j\} \cup K)$.

The starting point of our subsequent study of conditional independence is the following collection of results. Proofs can be found, for instance, in \S4.1 of \cite{sullivant2018algebraic}.

%
%Taking the $K \times K$ principal submatrix has a ``dual'' operation, the \emph{Schur complement} $\Sigma^K \defas \Sigma_{K^\co} - \Sigma_{K^\co,K} \Sigma_K^{-1} \Sigma_{K,K^\co}$ in~$\Sigma$. Its rows and columns are indexed by $K^\co = V \setminus K$. We will frequently use Schur's formula which relates the determinant of a matrix to that of a principal submatrix and its Schur complement: $|\Sigma| = |\Sigma_K| \cdot |\Sigma^K|$; see~\cite{Zhang} for an extensive account. Principal submatrices of Schur complements in positive~definite matrices are also~positive~definite.
%

% \begin{theorem} \label{thm:GaussianCI}
% Let $X \sim \mathcal{N}(\mu, \Sigma)$ be a multivariate Gaussian with $\Sigma \in \pd_V$ and let $K\subseteq V$.
% \begin{itemize}[itemsep=0.6em]
% \item The marginal vector $X_K = (X_k : k \in K)$ is a regular Gaussian in $\rr^K$ with mean vector $\mu_K$ and covariance~$\Sigma_K$.
% \item Given $y \in \rr^K$, the conditional $X_{K^\co} \mid X_K = y$ is a regular Gaussian in $\rr^{K^\co}$ with mean vector $\mu_{K^\co} + \Sigma_{K^\co,K} \Sigma_K^{-1} (y - \mu_K)$ and covariance $\Sigma^K$.
% \item Suppose $V = I \cup J$. Then~the marginal independence $X_I \indep X_J$ holds if and only if $\Sigma_{I,J} = 0$, i.e., $\Sigma$ is block-diagonal with an $I \times I$ and a $J \times J$ block.
% \end{itemize}
% \end{theorem}

\begin{theorem}\label{thm:GaussianCI} %\cite{sullivant2018algebraic}
Let $\Sigma\in\pd_n$.
The conditional independence (CI) statement $I \indep J \mid K$ holds in the Gaussian random vector $X \sim \cn(\mu, \Sigma)$ if and only if submatrix
$\Sigma_{IK, JK}$ has rank $|K|$. 
\end{theorem}

Since Gaussian distributions satisfy the compositional graphoid axioms, the CI statement $I \indep J \mid K$ holds if and only if $i \indep j \mid K$ holds for all $i \in I$ and $j \in J$ \cite{lauritzen2018unifying}. Thus, to characterize the set of CI statements that a Gaussian distribution satisfies, it suffices to characterize all statements of the form $i \indep j \mid K$. This is summarized with the following corollary, which combines this fact with the previous theorem. 

\begin{corollary}
\label{cor:GaussianCI}
%Let $X\sim \cn(\mu, \Sigma)$ with $\Sigma\in\pd_n$. 
For a Gaussian random vector $X\sim \cn(\mu, \Sigma)$ with $\Sigma\in\pd_n$, the following statements are equivalent:
\begin{enumerate}[itemsep=0.3em, label=(\roman*)]
\item $I \indep J \mid K$,
\item $i \indep j \mid K$ for all $i \in I$ and $j \in J$, and
\item $\det(\Sigma_{iK, jK}) \equiv |\Sigma_{iK, jK}| = 0$ for all $i \in I$ and $j \in J$.
\end{enumerate}
\end{corollary}

This characterization of conditional independence for multivariate Gaussians will be used extensively in \Cref{sec:LyapunovCI}.

\section{Conditional Independence in Lyapunov Models}
\label{sec:LyapunovCI}

Our focus is now on the 
%Recall that we consider the graphical continuous 
Lyapunov model of a directed graph $\cg = (V,E)$, which is given by the set of covariance matrices 
\[
\cm_{\cg} = \{\Sigma \in \pd_{\nrnodes} ~:~ M \Sigma + \Sigma M^T + 2 \Id_\nrnodes = 0  \text{ for a stable matrix } M \in \rr^E\}.
\]
%Since the distributions in $\cm_\cg$ are also given by the stationary %distribution of a $\cg$-compatible drift, 
\cref{thm:MarginalIndependence} applies to $\cm_\cg$ %as well 
and shows that the absence of a trek in $\cg$ leads to a marginal independence, and \cref{thm:Section2:CIImplied} yields the implied 
%. The global Markov property discussed in \cref{sec:ImpliedCI} then implies
additional conditional independence (CI)~statements for $\cm_\cg$.
The aim of this section is to prove that the resulting collection of CI~statements is \emph{complete}, by showing that for any other CI~statement, there is a distribution $\Sigma \in \cm_\cg$ which violates it.
More precisely, our main result is as follows.
%The marginal independence statements which correspond to the absence of a trek are actually the only non-trivial CI statements which hold for these models. Any other CI statement which holds for all covariance matrices $\Sigma \in \cm_\cg$ is implied by the marginal independence statements which come from \cref{thm:MarginalIndependence} as a result of the global Markov property discussed in \cref{sec:ImpliedCI}.
%Our main result in this section is the following theorem which formalizes this~idea. 

% condition "simple graph" removed - not necessary?
\begin{theorem}
\label{thm:CIImpliedByMarginals}
Let $\cg$ be a directed graph with trek graph $\hatg$. Then the Lyapunov model $\cm_\cg$ is Markov-perfect to the marginal independence graph $\hatg$. In other words, $i \indep j \mid K$ holds for all $\Sigma \in \cm_\cg$ if and only if  $V \setminus (\{i ,j \} \cup K)$ separates $i$ and $j$ in~$\hatg$. 
\end{theorem}

The proof of this theorem will be completed at the end of this section. As noted above, the direction to be proven is the fact that a CI statement holds for all $\Sigma \in \cm_\cg$ only if we have graphical separation.
% Observe that one direction, namely $V \setminus \{i ,j \} \cup K$ separating $i$ and $j$ implies that $i \indep j \mid K$, follows immediately from \Cref{thm:Section2:CIImplied}.

\subsection{Outline of the Proof}
%: Zig-Zag Graphs and a Restricted Trek Rule}

We will prove the contrapositive, % of the remaining direction,
i.e., we will show that if $V \setminus (\{i ,j \} \cup K)$ does not separate $i$ and $j$ in~$\hatg$, then there exists a matrix $\Sigma \in \cm_\cg$ such that $i \not \indep j \mid K$.  To do this, we first clarify that we may reduce the problem to studies of suitable subgraphs, and we introduce a trek rule which gives a polynomial parameterization of a submodel of $\cm_\cg$ in terms of the entries of $M$ (\cref{sec:sub-subgraphs-trek-rule}). We then show how to apply this restricted trek rule to the case where $\cg$ is a single trek in order to produce covariance matrices $\Sigma$ in $\cm_\cg$ such that the endpoints of the trek are conditionally dependent given all other nodes on the trek % for any nodes $i$ and $j$ on a trek are dependent given the remaining nodes $K$ by verifying that $|\Sigma_{iK, jK}| \neq 0$. 
Via a trick of creating a limiting scenario of \emph{perfect correlation}, we  extend this result to more general conditioning sets comprised of arbitrary subsets of the other nodes on the trek (\cref{sec:sub-CI-trek}).
%by making certain nodes \emph{perfectly correlated}. 
Finally, to complete the proof, we extend the previous results from a single trek to \emph{zig-zag} graphs, which correspond to paths that are connecting for the global Markov property described in the proof of \cref{thm:Section2:CIImplied} (\cref{sec:sub-CI-zig-zag}).  We note that these \emph{zig-zag} graphs also appear in the description of the d-separation criterion for structural equation models; cf.~\cite[Section~4.3]{MatusMinors}. %since the marginal independence structure is the same. 

\subsection{Subgraphs and a Restricted Trek Rule}
\label{sec:sub-subgraphs-trek-rule}

The following lemma clarifies how to pass to subgraphs for the construction of counterexamples to CI statements.

% requirement "simple" removed
\begin{lemma} \label{lemma:Specialization}
Let $\cg = (V, E)$ be a directed graph, and let $\cg' = (V', E')$ be a subgraph of $\cg$. 
If $i \not\indep j \mid K'$ on $\cm_{\cg'}$, then $i \not\indep j \mid K$ on $\cm_{\cg}$ for all $K \subseteq V$ with $K \cap V' = K'$.
\end{lemma}

\begin{proof}
The hypothesis that $i \not\indep j \mid K'$ means that there exists $\Sigma' \in \cm_{\cg'}$ for which ${|\Sigma'_{iK',jK'}| \neq 0}$. Let~$M'$ be the corresponding parameter matrix, i.e., $M'\Sigma' + \Sigma' {M'}^T + 2\,\Id_{V'} = 0$. Denote $V'' = V \setminus V'$. Then~the matrix equation with blocks indexed by $V'$ and $V''$
\[
  \underbrace{\begin{pmatrix}
    M' & 0 \\ 0 & -\Id_{V''}
  \end{pmatrix}}_{\asdef M}
  \underbrace{\begin{pmatrix}
    \Sigma' & 0 \\ 0 & \Id_{V''}
  \end{pmatrix}}_{\asdef \Sigma} + 
  \underbrace{\begin{pmatrix}
    \Sigma' & 0 \\ 0 & \Id_{V''}
  \end{pmatrix}}_{= \Sigma}
  \underbrace{\begin{pmatrix}
    {M'}^T & 0 \\ 0 & -\Id_{V''}
  \end{pmatrix}}_{= M^T} +
  \underbrace{\begin{pmatrix}
    2\,\Id_{V'} & 0 \\ 0 & 2\,\Id_{V''}
  \end{pmatrix}}_{= 2\,\Id_V} = 0
\]
shows that $\Sigma \in \cm_{\cg}$. Since $K \cap V' = K'$, it follows that $|\Sigma_{iK,jK}| = |\Sigma'_{iK',jK'}| \neq 0$.
\end{proof}

The previous lemma will be used throughout this section to reduce the problem of determining whether a conditional dependence holds to a minimal subgraph which exhibits it. The next proposition is then our major tool to study the resulting minimal subgraphs.
%which is used in the proof of \cref{thm:CIImpliedByMarginals} and many accompanying lemmas. 

\begin{proposition}[Restricted trek rule]
\label{prop:TrekRule}
Let $\cg= (V, E)$ be a DAG, and consider the slice of the model $\cm_{\cg}$ obtained by restricting the drift matrix $M=(m_{ij})$ to have diagonal entries $m_{ii} = \sfrac{-1}{2\zeta}$ for all $i \in V$. Then the following \emph{trek rule} holds for the entries of the covariance $\Sigma=(\sigma_{ij})$ associated to $M$:
\begin{equation}
  \label{eq:TrekRule}
  \sigma_{ij} = \sum_{\substack{T : \text{$(\ell, r)$-trek} \\ \text{from $i$ to $j$}}} 2 \zeta^{\ell+r+1} \binom{\ell+r}{\ell} m_T,
\end{equation}
where the \emph{trek monomial} $m_T$ associated to a trek $T$ is given by $m_T = \prod_{e \in T}m_e$, i.e., the product over all the edges $e$ in that trek.
\end{proposition}

\begin{proof}
%If there is no trek between $i$ and $j$, then the claimed expression in \eqref{eq:TrekRule} is zero which is true by \Cref{thm:MarginalIndependence}.
The Lyapunov equation $M\Sigma + \Sigma M^T = -C$ is linear in $\Sigma$, and 
% may be written in vectorized form $\widetilde{M} \vec(\Sigma) = -\vec(C)$, where $\vec(\Sigma)$ writes an $\nrnodes \times \nrnodes$ matrix as a vector of length $\nrnodes^2$ and $\widetilde{M} = \Id_V \otimes M + M \otimes \Id_V$ is a Kronecker sum whose rows and columns are indexed by pairs $ij$ with $(i,j) \in V \times V$; cf.~\cite[Eq.~(4.1)]{dettling2022identifiability}.
% Consider the $ij$-row of this equation:
% \[
%   \sum_{k, l} \widetilde{m}_{ij,kl} \sigma_{kl} = -C_{ij}.
% \]
% By definition of the Kronecker product, $\widetilde{m}_{ij,kl} \neq 0$ only in the following cases:
% \begin{itemize}
% \item if $k = i$ and $l = j$ then $\widetilde{m}_{ij,ij} = m_{ii} + m_{jj}$;
% \item if $k = i$ and $l \neq j$ then $\widetilde{m}_{ij,il} = m_{jl}$;
% \item if $k \neq i$ and $l = j$ then $\widetilde{m}_{ij,kj} = m_{ik}$.
% \end{itemize}
% The $ij$-row of the Lyapunov equation hence becomes
the $ij$-th element of the matrix equation reads
\[
  (m_{ii} + m_{jj}) \sigma_{ij} + \sum_{l \in \pa(j)} m_{jl} \sigma_{li} + \sum_{k \in \pa(i)} m_{ik} \sigma_{kj} = -C_{ij}.
\]
Note that the coefficient $m_{ii} + m_{jj}$ of $\sigma_{ij}$ is always $\sfrac{-1}{\zeta}$. Solving for $\sigma_{ij}$, we obtain
\begin{align}
  \label{eq:VecLyapunov}
  \sigma_{ij} = \zeta\cdot\left(C_{ij} + \sum_{l \in \pa(j)} m_{jl} \sigma_{li} + \sum_{k \in \pa(i)} m_{ik} \sigma_{kj}\right).
\end{align}
These relations motivate a proof of \eqref{eq:TrekRule} by induction. 

We fix a topological ordering of the DAG~$\cg$ and extend it to a lexicographic ordering on the entries of $\Sigma$ via $\sigma_{ij} < \sigma_{kl}$ if and only if $j < l$ or $j = l$ and $i < k$. We perform induction with respect to this ordering, noting that the right-hand side of \eqref{eq:VecLyapunov} contains only $\Sigma$-entries which are strictly smaller than $\sigma_{ij}$ on the left-hand side.
The smallest $\Sigma$-entries are $\sigma_{ii}$ where $i$ is a source node. In this case $\sigma_{ii} = \zeta C_{ii} = 2\zeta$ follows immediately from the Lyapunov equation for our choice of $C = 2\Id_V$. This verifies the claim as there is only the empty trek on~$i$. \Cref{thm:MarginalIndependence} or equation~\eqref{eq:VecLyapunov} can be used to see that if $i$ and $k$ are distinct source nodes (having in particular no treks between them), then $\sigma_{ki} = 0$. Now let $j$ be a node all of whose parents are source nodes; the edges into such nodes $j$ form the next smallest class of $\Sigma$-entries in the lexicographic ordering. The Lyapunov equation simplifies to $\sigma_{ij} = \zeta \sum_{k \in \pa(j)} m_{jk} \sigma_{ki} = \zeta m_{ji} \sigma_{ii} = 2 \zeta^2 m_{ji}$ which corresponds to the unique trek between $i$ and~$j$ of length $\ell+r = 1$. These two cases prove the initial induction hypothesis.

Fix any two (not necessarily distinct) vertices $i, j$ and a non-empty $(\ell,r)$-trek $T$ between them. This trek decomposes into a unique $(\ell,r-1)$-trek $T_1$ from $i$ to $l$, for $l \in \pa(j)$, and the edge $l \to j$. It also decomposes uniquely into an $(\ell-1,r)$-trek $T_2$ from $k$ to $j$, for $k \in \pa(i)$, together with the edge $k \to i$. By \eqref{eq:VecLyapunov}, the coefficient of $m_T$ in $\sigma_{ij}$ is thus $\zeta$ times the sum of the coefficients of $m_{T_1}$ in $\sigma_{il}$ and $m_{T_2}$ in $\sigma_{kj}$. Using the induction hypothesis (since $\sigma_{il}$ and $\sigma_{kj}$ are both less than $\sigma_{ij}$) gives
\[
  \text{coefficient of $m_T$} = \zeta\cdot\left(2 \zeta^{\ell+r} \binom{\ell+r-1}{\ell} + 2 \zeta^{\ell+r} \binom{\ell+r-1}{\ell-1}\right) = 2 \zeta^{\ell+r+1} \binom{\ell+r}{\ell},
\]
as claimed. This immediately proves the induction step when $i \neq j$. In case $i = j$, there is also the empty trek from $i$ to $i$ to account for and its trek term arises as required from \eqref{eq:VecLyapunov} via $C_{ii} = 2$.
\end{proof}

\begin{figure}
\centering
\begin{tikzpicture}
\node [draw, circle] (i) {\strut$2$};
\node [above right = 1 of i, draw, circle, radius = .2cm] (t1) {\strut$1$};
\node [below right = 1 of t1, draw, circle] (k1) {\strut$3$};
\node [below = 2 of t1, draw, circle] (l1) {\strut$4$};

\draw [->] (t1) -- (i);
\draw [->] (t1) -- (k1);
\draw [->] (k1) -- (l1);
\draw [->] (i) -- (l1);

\end{tikzpicture}
\caption{The diamond graph discussed in \Cref{ex:TrekRule}.}
\label{fig:TrekRuleExample}
\end{figure}
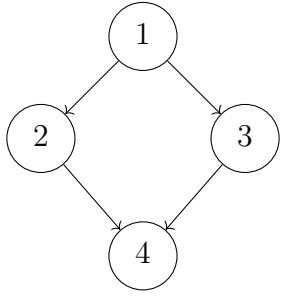

\begin{example}
\label{ex:TrekRule}
Consider the graph in \Cref{fig:TrekRuleExample}. Using the trek rule formulation for $\sigma_{ij}$ from \cref{prop:TrekRule}, we get 
\begin{eqnarray*}
    \sigma_{11} &=& 2\zeta, \\
    \sigma_{22} &=& 4\zeta^3m_{21}^2 +2\zeta, \\
    \sigma_{33} &=& 4\zeta^3m_{31}^2 +2\zeta, \\
    \vdots \\
    \sigma_{34} &=& 6\zeta^4m_{21}m_{31}m_{42} + 6\zeta^4m_{31}^2m_{43} + 2\zeta^2m_{43}.
\end{eqnarray*}
For instance, observe that there are two treks from the vertex $2$ to itself, the empty trek (corresponding to the monomial $2\zeta$ for $\sigma_{22}$) and the trek $2\leftarrow 1 \rightarrow 2$ ($4\zeta^3m_{21}^2$). Similarly, there are three treks between the vertices $3$ and $4$, which are the edge $3\rightarrow 4$ ($2\zeta^2m_{43}$), the trek $2\leftarrow 1 \rightarrow 3 \rightarrow 4$ ($6\zeta^4m_{21}m_{31}m_{42}$), and $3\leftarrow 1 \rightarrow 3 \rightarrow 4 $ ($6\zeta^4m_{31}^2m_{43}$). For the last trek $3\leftarrow 1 \rightarrow 3 \rightarrow 4$, we have $\ell=1$ and $r=2$, which gives us the coefficient $6\zeta^4$.
\end{example}

This trek rule is reminiscent of the original trek rule for linear structural equation models \cite{draisma2010trek,sullivant2008algebraic}. However, now each trek monomial is weighted by a binomial coefficient. While the trek rule we present here only parameterizes a proper subset of $\cm_\cg$ and is less general than that of \cite{varando2020graphical}, it provides an explicit rational formula for the entries of $\Sigma$ in terms of the parameters $M$ which makes it significantly easier to construct explicit covariance matrices in $\cm_\cg$ with the desired dependence properties. We remark that the very recent preprint of \cite{Hansen2024} gives a generalization of \Cref{prop:TrekRule}.

\subsection{Conditional Independence on Treks}
\label{sec:sub-CI-trek}

We now turn to the case where the considered graph is a single trek, for which we show that the associated Lyapunov model does not satisfy any CI~relations.
Before proceeding, we note that if $\ct$ is a trek from node 1 to $\nrnodes$ with top node $t$, then  a drift matrix $M\in\mathbb{R}^{E(\ct)}$ 
    %inducing a distribution in the corresponding Lyapunov model of $\ct$ 
    has zero-pattern
    \begin{equation}
    \label{eq:trek-drift-matrix}
    M = \left(\begin{array}{ccccccc}
        d_1 & * & & & & & \\
        & \ddots & \ddots & & & & \\
        & & d_{t-1} & * & & & \\
        & & & d_t & & & \\
        & & & * & d_{t+1} & & \\
        & & & & \ddots & \ddots & \\
        & & & & & * & d_\nrnodes \\
    \end{array}\right)
    \end{equation}
    with diagonal entries $d_1,\dots,d_\nrnodes$ and the symbol $*$
    indicating arbitrary values.  Such matrices are stable if and only
    if the diagonal entries are negative.
    %as can be seen by computing their characteristic polynomial.

% In the setting of linear SEMs, the trek rule expressing the parameterization of the model is complemented by a \emph{trek separation criterion} \cite{draisma2010trek} which determines the rank-deficient submatrices in the parametrization and hence the conditional independence statements which hold on the model. However, our goal in the Lyapunov setting is to show that there are no conditional independencies beyond what is implied by marginal independence. Hence, the first step is to handle the trek case and show that treks do not satisfy any CI~statements.

\begin{proposition}
\label{Prop:TrekDependence}
Let $\ct$ be an $(\ell,r)$ trek between $i$ and $j$. Let $K = V(\ct) \setminus \{i, j\}$ be the remaining vertices. Then $i \not \indep j \mid K$ in $\cm_\ct$. 
\end{proposition}

The proof of this proposition and its several supporting lemmas can be found in \cref{appendix:TrekDependence}. The following example provides an explicit family of covariance matrices in the Lyapunov model for a small trek which exhibit the conditional dependence stated in \Cref{Prop:TrekDependence}. The proof of the previous proposition generalizes this family of covariance matrices to larger treks and shows that the conditional dependence still holds by inductively computing certain entries in the adjoint matrix of $\Sigma_{iK, jK}$. The Schur complement expansion of $|\Sigma_{iK, jK}|$ with respect to the principal minor $\Sigma_{K,K}$ can then be used to show that $|\Sigma_{iK, jK}| \neq 0$ by inductively computing the lowest degree term which appears in $|\Sigma_{iK, jK}|$ when it is viewed as a polynomial in $\zeta$.

\begin{figure}
\centering
\begin{tikzpicture}
\node [draw, circle] (i) {\strut$2$};
\node [above right = 1 of i, draw, circle, radius = .2cm] (t1) {\strut$1$};
\node [below right = 1 of t1, draw, circle] (k1) {\strut$3$};
\node [above right = 1 of k1, draw, circle] (l1) {\strut$4$};

\draw [->] (t1) -- (i);
\draw [->] (t1) -- (k1);
\draw [->] (k1) -- (l1);

\node [below = .5 of k1] {$\ct$};

\node [right = 2 of l1] {};
\end{tikzpicture}
\begin{tikzpicture}
\node [draw, circle] (i) {\strut$3$};
\node [above right = 1 of i, draw, circle, radius = .2cm] (t1) {\strut$1$};
\node [below right = 1 of t1, draw, circle] (k1) {\strut$4$};
\node [above right = 1 of k1, draw, circle] (l1) {\strut$2$};
\node [below right = 1 of l1, draw, circle] (j) {\strut$5$};

\draw [->] (t1) -- (i);
\draw [->] (t1) -- (k1);
\draw [->] (l1) -- (k1);
\draw [->] (l1) -- (j);

\node [below= .5 of k1] {$\cz$};
\end{tikzpicture}
\caption{A trek $\ct$ which exhibits no conditional independence statements and a \emph{zig-zag} graph $\cz$ which exhibits some marginal independencies.}
\label{fig:Trek with no CI}
\end{figure}

\begin{example}
\label{ex:Trek with no CI statements}
Let $\ct$ be the trek as shown in \Cref{fig:Trek with no CI}. Setting every edge parameter $m_{ji}$ to $1$ for all $i\rightarrow j \in E(\ct)$,
%(as seen in the proof of \cref{Prop:TrekDependence}), 
we get the following covariance matrix from the trek rule which lies in the model:
\[
\Sigma=\begin{pmatrix}
    2\zeta & 2\zeta^2 & 2\zeta^2 & 2\zeta^3 \\
    2\zeta^2 & 4\zeta^3+2\zeta & 4\zeta^3 & 6\zeta^4 \\
    2\zeta^2 & 4\zeta^3 & 4\zeta^3+2\zeta & 6\zeta^4+2\zeta^2 \\
    2\zeta^3 & 6\zeta^4 & 6\zeta^4+2\zeta^2 & 12\zeta^5+4\zeta^3+2\zeta 
\end{pmatrix}.
\]
Observe that the submatrix $\Sigma_{\{1,2,3\},\{2,3,4\}}$ has determinant $-8\zeta^6$, which is non-zero for $\zeta\not=0$, implying that $2 \not \indep 4 \mid \{1,3\}$ in the model $\cm_\ct$.
\end{example}

We now introduce a technique to extend the previous proposition to all conditioning sets $K \subseteq V(\ct) \setminus \{i, j\}$.  In Lyapunov models with fixed diffusion matrix $C=2 \Id_V$, we cannot simply pass to marginal distributions as, in 
% As the distributions in the model are parametrized by the continuous Lyapunov equation via appropriate drift and diffusion matrices, we cannot simply marginalize out nodes. In 
general, removing a row and corresponding column in a covariance matrix $\Sigma\in\cm_\ct$ does not yield
%the covariance matrix of a distribution in the Lyapunov model by a trek does not yield 
a distribution in the Lyapunov model over the shortened trek in which the removed node is skipped.
%graph subgraph. 
Instead, we provide a construction in which the random variables indexed by nodes in $V(\ct) \setminus K$ become perfectly correlated with random variables indexed by elements of $K$.  This construction then allows us to argue that the counterexample from \cref{Prop:TrekDependence} can be `lifted' to counterexamples on longer treks between $i$ and $j$ that contain additional nodes not in the conditioning set $K$.

% construct a covariance matrix $\Sigma$ as the solution of a Lyapunov equation defined via an explicit drift matrix $M$.
% We start with a distribution where independence conditioned on all possible nodes does not hold.
% The key idea is to ensure additional nodes that are not contained in the conditioning set $K$ are (nearly) perfectly correlated with existing conditioning nodes on the trek, thereby preserving the dependence structure in the considered distribution.

% For a node $k$ with $\pa(k)\defas\Pa(k)\setminus\{k\}$ a singleton set, we identify $\pa(k)$ with its single element, i.e., the unique node $j$ with $j\to k$. Similarly, if $\ch(k)\defas\Ch(k)\setminus\{k\}$ is a singleton, we identify $\ch(k)$ with the unique node $j$ such that $k\to j$.

To introduce the construction, we recall the following notation. 
For a node $k$ for which $\pa(k)$ is a singleton set, we identify $\pa(k)$ with its single element, i.e., the unique node $j$ with $j\to k$. Similarly,  we identify $\ch(k)$ with the unique node $j$ such that $k\to j$.  Moreover, for $k \in \nn$, we define a map that shifts indices as
\begin{equation}
  \label{eq:index-shift-map}
    \psi_{-k}: 
    \nn \setminus \{k\} \to \nn, \quad 
    x \mapsto \begin{cases}
        x &\text{if } x < k; \\
        x-1 &\text{if } x > k.
    \end{cases} 
\end{equation}

Now, consider a trek $\ct$ between $i$ and $j$ with top node $t$ and of length $|V(\ct)|\ge 4$. If $|\ch(t)| = 2$, let $c_l$ and $c_r$ be the child nodes on the left and on the right side of the trek, respectively. Let $k \in V(\ct)\setminus\{i,j\}$ be a node on the trek, and consider the subset of nodes $V_{-k} = V(\ct)\setminus \{k\}$ obtained by removing $k$.  Define a shorter trek 
\begin{equation*}
    \ct_{-k} = (V_{-k},E_{-k})
\end{equation*}
% where the vertex set $V_{-k} = V(\ct)\setminus \{k\}$ is obtained by removing the node $k$ and the edge set $E_{-k}$ is constructed as follows.
by constructing the edge set $E_{-k}$ as follows.
If $k\not=t$,  replace the 2-path $\pa(k) \to k \to \ch(k)$ by the single edge $\pa(k)\to \ch(k)$. If $k=t$, the trek $\ct$ has two non-trivial sides as $k\notin\{i,j\}$.
%, as otherwise $k \in \{i,j\}$. 
In this case, if the right side of $\ct$ has length 1, we replace $c_l \leftarrow k \to c_r$ by $c_l\to c_r$. If $k = t$ and the right side consists of more than one edge, we replace $c_l \leftarrow k \to c_r$ by $c_l\leftarrow c_r$. The construction is depicted in \Cref{fig:ConstructionTMinusK}.
In all cases, the resulting graph is shorter by one node, as the $k$-th node is removed. Note that a similar definition is possible by adding the reverse edge $c_l \to c_r$ in the case where $k = t$ and both sides of the trek consist of more than one edge. The distinction between the second and third case is necessary to keep the trek structure when removing a top node from a trek where one side consists of exactly one edge. Otherwise, the resulting graph could be a directed path, e.g., $\ct_{-k}$ could be $i \ot c_l \ot c_r$ in \Cref{fig:ConstructionTMinusK:b}. 
In the case that the nodes $V(\ct) = [\nrnodes]$ are numbered consecutively from left to right, we implicitly apply $\psi_{-k}$ to the elements of $V_{-k}$ and the start- and endpoint of each edge in $E_{-k}$ so that we have $V_{-k} = [\nrnodes-1]$ and $E_{-k} \subseteq [\nrnodes-1] \times [\nrnodes-1]$.

\begin{figure}
    \newcommand{\scalefactor}{0.8}
    \centering
    \begin{subfigure}[t]{0.3\textwidth}
        \centering
        \scalebox{\scalefactor}{%
        \begin{tikzpicture}[minimum size = 1cm]
        \node [draw, circle] (i) {\strut$i$};
        \node [above right = 1 of i, draw, circle, blue] (k) {\strut$k$};
        \node [above right = 1 of k, draw, circle] (t) {\strut$t$};
        \node [below right = 1 of t, draw, circle] (j) {\strut$j$};
        \node [scale = 1/\scalefactor] (T) at (-0.5,2.5) {\strut$\ct$};
        \node [below right = 0.5 and 1.7 of i, draw, circle] (tm) {\strut$t$};
        \node [below left = 1 of tm, draw, circle] (im) {\strut$i$};
        \node [below right = 1 of tm, draw, circle] (jm) {\strut$j$};
        \node [scale = 1/\scalefactor] (Tminus) at (-0.5,-1.5) {\strut$\ct_{-k}$};
        
        \draw [->, blue] (k) -- (i);
        \draw [->, blue, dashed] (t) -- (k);
        \draw [->] (t) -- (j);
        \draw [->, blue] (tm) -- (im);
        \draw [->] (tm) -- (jm);
        \end{tikzpicture}}
    \caption{Case $t \neq k$ with $\ch(k) = \{i\}$ and $\pa(k) = \{t\}$.}
    \label{fig:ConstructionTMinusK:a}
    \end{subfigure}%
    \hfill
    \begin{subfigure}[t]{0.3\textwidth}
        \centering
        \scalebox{\scalefactor}{%
        \begin{tikzpicture}[minimum size = 1cm]
        \node [draw, circle] (i) {\strut$i$};
        \node [above right = 1 of i, draw, circle] (cl) {\strut$c_l$};
        \node [above right = 1 of cl, draw, circle, blue] (k) {\strut$k$};
        \node [below right = 1 of k, draw, circle] (cr) {\strut$c_r$};
        \node [below right = 0.5 and 1.7 of i, draw, circle] (clm) {\strut$c_l$};
        \node [below left = 1 of clm, draw, circle] (im) {\strut$i$};
        \node [below right = 1 of clm, draw, circle] (crm) {\strut$c_r$};
        
        \draw [->, blue, dashed] (k) -- (cl);
        \draw [->] (cl) -- (i);
        \draw [->, blue] (k) -- (cr);
        \draw [->] (clm) -- (im);
        \draw [->, blue] (clm) -- (crm);
        \end{tikzpicture}}
    \caption{Case $t = k$ with right side exactly length 1.}
    \label{fig:ConstructionTMinusK:b}
    \end{subfigure}%
    \hfill
    \begin{subfigure}[t]{0.3\textwidth}
        \centering
        \scalebox{\scalefactor}{%
        \begin{tikzpicture}[minimum size = 1cm]
        \node [draw, circle] (cl) {\strut$c_l$};
        \node [above right = 1 of cl, draw, circle, blue] (k) {\strut$k$};
        \node [below right = 1 of k, draw, circle] (cr) {\strut$c_r$};
        \node [below right = 1 of cr, draw, circle] (j) {\strut$j$};
        \node [below left = 0.5 and 1.7 of j, draw, circle] (crm) {\strut$c_r$};
        \node [below left = 1 of crm, draw, circle] (clm) {\strut$c_l$};
        \node [below right = 1 of crm, draw, circle] (jm) {\strut$j$};
        
        \draw [->, blue] (k) -- (cl);
        \draw [->, blue, dashed] (k) -- (cr);
        \draw [->] (cr) -- (j);
        \draw [->, blue] (crm) -- (clm);
        \draw [->] (crm) -- (jm);
        \end{tikzpicture}}
    \caption{Case $t = k$ with ride side longer than length 1.}
    \label{fig:ConstructionTMinusK:c}
    \end{subfigure}
    \caption{Construction of $\ct_{-k}$ depending on the position of the node $k$ in $\ct$. The node $k$ and its in- and outgoing edges (blue in $\ct$) are replaced by a single edge oriented accordingly (blue in $\ct_{-k}$). The dashed edge indicates the two nodes that will be perfectly correlated in the following construction of a distribution in $\cm_{\ct}$ given a distribution in $\cm_{\ct_{-k}}$.}
    \label{fig:ConstructionTMinusK}
\end{figure}
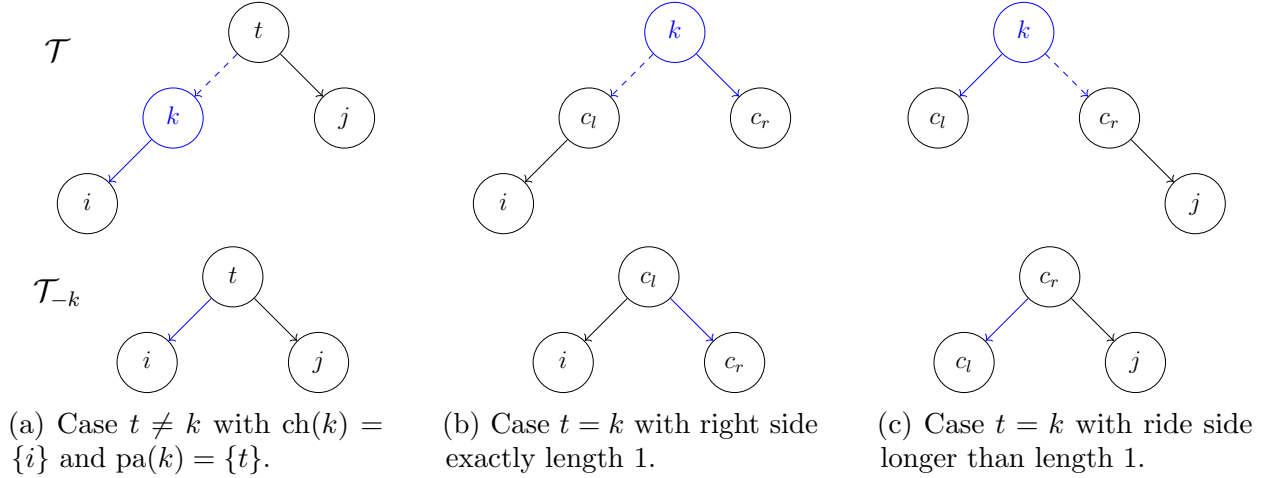

Given a matrix $\Sigma$, we write 
\[
    \Sigma_{-k,-k} = \Sigma_{[\nrnodes]\setminus \{k\},[\nrnodes]\setminus\{k\}}
\]
for the submatrix obtained by removing the $k$\textsuperscript{th} row and column of $\Sigma$.
Given a set of matrices $\cm \subseteq \rr^{\nrnodes \times \nrnodes}$, we denote its topological closure by $\overline{\cm}$.

\begin{lemma} \label{lemma:PerfectCorrelationMatrix}
    Let $\ct$ be a trek between $i$ and $j$ containing more than 3 nodes.  Let $t$ be the top node of $\ct$, and let $k \notin \{i,t,j\}$ be a further node on $\ct$.  Then for every matrix $\Sigma^* \in \cm_{\ct_{-k}}$, there exists a matrix $\Sigma \in \overline{\cm_{\ct}}$ such that 
    \begin{equation*}
        \Sigma_{-k,-k} = \Sigma_{-\pa(k),-\pa(k)} = \Sigma^*.
    \end{equation*}
\end{lemma}

\begin{proof}
    Let $\nrnodes \geq 4$ be the number of nodes on the trek $\ct$.
    Without loss of generality, assume that the nodes on $\ct$ are numbered consecutively from left to right, i.e., $i = 1$, $j = \nrnodes$, and $k \in \{2,\dots,\nrnodes-1\}\setminus \{t\}$. 

    From $\Sigma^* \in \cm_{\ct_{-k}}$, we are to construct a suitable
    larger covariance matrix $\Sigma\in\pd_\nrnodes$.  To do this,
    we make the
    new node at index $k$ a copy of its parent on the
    trek by duplicating the respective row and column from $\Sigma^*$. In the resulting distribution,
    % the node $k$ has the same dependence relations with the remaining
    % nodes as its parent $\pa(k)$. In particular, the nodes
    $k$ and $\pa(k)$ are perfectly correlated. 
    Note that $\pa(k) = \{k+1\}$ for $k < t$, and $\pa(k)=\{k-1\}$ for $k > t$. Thus, in the first case, the $k$\textsuperscript{th} and $k+1$\textsuperscript{st} row and column of $\Sigma$ are duplicates of each other, while in the second case, the $k$\textsuperscript{th} and $k-1$\textsuperscript{st} row and column of $\Sigma$ are duplicates of each other. The proof in both cases is analogous as the drift matrix of a single trek is defined in a symmetric way around its top node as seen in \eqref{eq:trek-drift-matrix}. We only give the proof for the second case, so assume $k > t$.
    Furthermore, let $U:=\{1,\dots,k-1\}$,
    $W:=\{k+1,\dots,\nrnodes\}$, and write $\psi \equiv \psi_{-k}$ for
    the map from~\eqref{eq:index-shift-map}. 
    Then, $\psi(U) = U$ and $\psi(W) = \{k, \dots, \nrnodes-1\}$.
    %We also use the abbreviated notation for principal submatrices introduced above.
    We note that the case of a directed path, i.e., a trek with
    $i=t=1$, is included here as $k > 1$ by assumption.  

    Given $\Sigma^* \in \cm_{\ct_{-k}}$, there exists a stable matrix $M^* \in \rr^{(\nrnodes-1)\times(\nrnodes-1)}$ such that the Lyapunov equation 
    \begin{equation} \label{eq:lyapunov_M_star}
        M^* \Sigma^* + \Sigma^* (M^*)^T + 2\Id_{\nrnodes-1} = 0
    \end{equation}
    is fulfilled.
    In the following, we show that our construction of inserting a
    duplicate of the row and column of $\Sigma^*$ at index $\pa(k)$ in
    $\ct$ indeed yields a matrix $\Sigma \in \overline{\cm_{\ct}}$.

    % by taking the covariance matrix $\Sigma^*$ and inserting at index $k$ a duplicate of the row and column at index $\pa(k)$ in $\ct$, yielding the node $k$ perfectly correlated with its parent node. 

    To achieve this, we construct a sequence of drift matrices $M^{(m)}$ as follows. First, partition
    \[
    \Sigma^* = 
    \left(\begin{array}{@{}c | c@{}} 
        \multirow{2}{*}{$\Sigma^*_{U}$} & \multirow{2}{*}{$\Sigma^*_{U,\psi(W)}$} \\
        & \\
        \hline
        \multirow{2}{*}{$\Sigma^*_{\psi(W),U}$} & \multirow{2}{*}{$\Sigma^*_{\psi(W)}$} \\
        & \\
    \end{array}\right)
    \]
    as well as 
    \[
    M^* = 
    \left(\begin{array}{@{}c | c@{}} 
        \multirow{2}{*}{$M^*_{U}$} & \multirow{2}{*}{$\mathbf{0}$} \\
        & \\
        \hline
        \multirow{2}{*}{$M^*_{\psi(W),U}$} & \multirow{2}{*}{$M^*_{\psi(W)}$} \\
        & \\
    \end{array}\right),
    \quad \text{where} \quad 
    M^*_{\psi(W),U} = 
    \left(\begin{array}{@{}c c@{}} 
        \begin{matrix}
            0 & \cdots & 0 \\
        \end{matrix} & \begin{matrix}
            m^*_{k,k-1} \\
        \end{matrix} \\
        \mathbf{0} & 
        \begin{matrix}
            0 \\ \vdots \\ 0 \\
        \end{matrix} \\
    \end{array}\right). 
    \]
    Considering \eqref{eq:lyapunov_M_star} block-wise results in the following four equations:
    \begin{enumerate}[itemsep=0.6em, label=(\alph*)]
        \item $M^*_{U} \Sigma^*_{U} + \Sigma^*_{U} \tp[-0]{M^*_{U}} + 2\Id_{k-1} = 0$, \label{eq:M_star:a}
        \item $M^*_{U} \Sigma^*_{U,\psi(W)} + \Sigma^*_{U,\psi(W)} \tp{M^*_{\psi(W)}} + \Sigma^*_{U} \tp{M^*_{\psi(W),U}} = 0$, \label{eq:M_star:b}
        \item $M^*_{\psi(W)} \Sigma^*_{\psi(W)} + \Sigma^*_{\psi(W)} \tp{M^*_{\psi(W)}} + 2I_{\nrnodes-k} + M^*_{\psi(W),U}\Sigma^*_{U,\psi(W)} + \Sigma^*_{\psi(W),U} \tp{M^*_{\psi(W),U}} = 0$, \label{eq:M_star_c}
        \item $M^*_{\psi(W)} \Sigma^*_{\psi(W),U} + \Sigma^*_{\psi(W),U} \tp[-0]{M^*_{U}} + M^*_{\psi(W),U} \Sigma^*_{U} = 0$, \label{eq:M_star:d}
    \end{enumerate}
    where equations \ref{eq:M_star:b} and \ref{eq:M_star:d} are transposes of each other.
    Given the stable matrix $M^*$, these equations are uniquely fulfilled by the corresponding submatrices of $\Sigma^*$.
    
    Now, let $m \in \nn_{>0}$ and define
    \[\renewcommand*{\arraystretch}{1.5}
    M^{(m)} = 
    \left(\begin{array}{@{}c c|c|c c@{}}
        \multicolumn{2}{c|}{\multirow{2}{*}{$M^*_{U}$}} & \multirow{2}{*}{$\mathbf{0}$} &  \multicolumn{2}{c}{\multirow{2}{*}{$\mathbf{0}$}} \\
        \multicolumn{2}{c|}{} &  & \multicolumn{2}{c}{} \\
        \hline
        \mathbf{0} & m & -m & \multicolumn{2}{c}{\mathbf{0}} \\
        \hline
        \multicolumn{2}{c|}{\multirow{2}{*}{$\mathbf{0}$}}  & m^*_{k,k-1} & \multicolumn{2}{c}{\multirow{2}{*}{$M^*_{\psi(W)}$}} \\
        \multicolumn{2}{c|}{} & \mathbf{0} & \multicolumn{2}{c}{} \\
    \end{array}\right) \in \rr^{\nrnodes\times \nrnodes}.
    \]
    Further, let
    \[\renewcommand*{\arraystretch}{1.5}
    \Sigma^{(m)} = 
    \left(\begin{array}{c c|c|c c}
        \multicolumn{2}{c|}{\multirow{2}{*}{$\Sigma^{(m)}_{U}$}} & \multirow{2}{*}{$\Sigma^{(m)}_{U,k}$} &  \multicolumn{2}{c}{\multirow{2}{*}{$\Sigma^{(m)}_{U,W}$}} \\
        \multicolumn{2}{c|}{} &  & \multicolumn{2}{c}{} \\
        \hline
        \multicolumn{2}{c|}{\Sigma^{(m)}_{k,U}} & \Sigma^{(m)}_{k,k} & \multicolumn{2}{c}{\Sigma^{(m)}_{k,W}} \\[1mm]
        \hline
        \multicolumn{2}{c|}{\multirow{2}{*}{$\Sigma^{(m)}_{W,U}$}}  & \multirow{2}{*}{$\Sigma^{(m)}_{W,k}$} & \multicolumn{2}{c}{\multirow{2}{*}{$\Sigma^{(m)}_{W}$}} \\
        \multicolumn{2}{c|}{} & & \multicolumn{2}{c}{} \\
    \end{array}\right) \in \rr^{\nrnodes \times \nrnodes}
    \]
    be the solution of the resulting Lyapunov equation 
    \begin{equation} \label{eq:LyapunovMm}
        M^{(m)} \Sigma^{(m)} + \Sigma^{(m)} (M^{(m)})^T + 2\Id_\nrnodes = 0
    \end{equation}
    induced by $M^{(m)}$ for every $m \in \nn_{>0}$. 
    
    Due to the symmetry of the Lyapunov equation, any unique solution of \eqref{eq:LyapunovMm} is symmetric. Since $M^{(m)}$ is stable, the solution $\Sigma^{(m)}$ is unique and therefore symmetric, so we only have to solve six matrix equations that arise by block-matrix multiplication, denoted by letters in brackets as follows:
    \begin{equation*} 
        \renewcommand*{\arraystretch}{1.5}
        M^{(m)} \Sigma^{(m)} + \Sigma^{(m)} (M^{(m)})^T + 2\Id_\nrnodes = 0 \equiv 
        \left(\begin{array}{c c|c|c c}
            \multicolumn{2}{c|}{\multirow{2}{*}{\text{(A)}}} & \multirow{2}{*}{\text{(B)}} &  \multicolumn{2}{c}{\multirow{2}{*}{\text{(C)}}} \\
            \multicolumn{2}{c|}{} &  & \multicolumn{2}{c}{} \\
            \hline
            \multicolumn{2}{c|}{} & \text{(D)} & \multicolumn{2}{c}{\text{(E)}} \\[1mm]
            \hline
            \multicolumn{2}{c|}{\multirow{2}{*}{}}  & \multirow{2}{*}{} & \multicolumn{2}{c}{\multirow{2}{*}{\text{(F)}}} \\
            \multicolumn{2}{c|}{} & & \multicolumn{2}{c}{} \\
        \end{array}\right).
    \end{equation*}

    The upper left block $\Sigma^{(m)}_U$ of the partitioned matrix
    $\Sigma^{(m)}$ is constant as it is the unique solution of
    equation \ref{eq:M_star:a}. Proceeding recursively, the other
    blocks of $\Sigma^{(m)}$ can be shown to converge for $m \to
    \infty$ as well, with limits given by the relevant parts of $\Sigma^*$. The details of these calculations can be found in \Cref{appendix:PerfectCorrelation}.
    Combining the results, we have
    \[\renewcommand*{\arraystretch}{1.5}
    \lim_{m \to \infty} \Sigma^{(m)} = 
    \left(\begin{array}{c c|c|c c}
        \multicolumn{2}{c|}{\multirow{2}{*}{$\Sigma^*_{U}$}} & \multirow{2}{*}{$\Sigma^*_{U,\pa(k)}$} &  \multicolumn{2}{c}{\multirow{2}{*}{$\Sigma^*_{U,\psi(W)}$}} \\
        \multicolumn{2}{c|}{} &  & \multicolumn{2}{c}{} \\
        \hline
        \multicolumn{2}{c|}{\Sigma^*_{\pa(k),U}} & \Sigma^*_{\pa(k),\pa(k)} & \multicolumn{2}{c}{\Sigma^*_{\pa(k),\psi(W)}} \\[1.5mm]
        \hline
        \multicolumn{2}{c|}{\multirow{2}{*}{$\Sigma^*_{\psi(W),U}$}}  & \multirow{2}{*}{$\Sigma^*_{\psi(W),\pa(k)}$} & \multicolumn{2}{c}{\multirow{2}{*}{$\Sigma^*_{\psi(W)}$}} \\
        \multicolumn{2}{c|}{} & & \multicolumn{2}{c}{} \\
    \end{array}\right) \in \overline{\cm_{\ct}},
    \]
as desired.
\end{proof}

\begin{corollary} \label{cor:PerfectCorrelationDeterminant}
    Let $\ct$ be a trek with endpoints $i$ and $j$. Let $\emptyset \neq K \subsetneq V(\ct)\setminus\{i,j\}$ and let $k \notin \{i,j\} \cup K$ be a node on the trek. If there is 
    \[
        \Sigma^* \in \cm_{\ct_{-k}} \text{ such that } | \Sigma^*_{\psi_{-k}(iK),\psi_{-k}(jK)}| \neq 0,
    \]
    then there is 
    \[
        \Sigma \in \cm_{\ct} \text{ such that } |\Sigma_{iK,jK}| \neq 0.
    \]
\end{corollary}

\begin{proof}
    By continuity of the determinant, it suffices to construct a matrix $\widehat{\Sigma} \in \overline{\cm_{\ct}}$ such that $ | \widehat{\Sigma}_{iK,jK}| \neq 0$.  Considering the closure of the model is crucial as it allows us to apply \cref{lemma:PerfectCorrelationMatrix} for the construction of $\widehat{\Sigma}$.   Then as $k \notin \{i,j\} \cup K$, we find that 
    \[
        | \widehat{\Sigma}_{iK,jK}| 
        = | \Sigma^*_{\psi_{-k}(iK),\psi_{-k}(jK)} |
         \neq 0. \qedhere
    \]
\end{proof}

\subsection{Conditional Independence in Zig-Zag Graphs}
\label{sec:sub-CI-zig-zag}

We now introduce the family of graphs that will be the key element needed to complete the proof of \cref{thm:CIImpliedByMarginals}. 

\begin{definition}
Let $\ct_1, \ldots, \ct_p$ be treks such that for $i \in [p-1]$, there exists a node $k_i$ on the right branch of $\ct_i$ which is the endpoint of the left branch of $\ct_{i+1}$. A $\emph{zig-zag}$ graph $\cz$ (\Cref{fig:ZigZagGraph}) is a directed acyclic graph obtained by performing an iterated 1-clique sum of these treks along the nodes $k_i$, i.e.
\[
\cz = \ct_1 \oplus_{k_1} \ct_{2} \oplus_{k_2} \cdots \ct_{p-1} \oplus_{k_{p-1}} \ct_{p}.
\]

\end{definition}

\begin{figure}
\centering
\begin{subfigure}{1\textwidth}
\centering%
\begin{tikzpicture}
\node [draw, circle] (i) {\strut$i$};
\node [above right = 1 of i, draw, circle, radius = .2cm] (t1) {\strut$t_1$};
\node [below right = 1 of t1, draw, circle] (k1) {\strut$k_1$};
\node [below = 1 of k1, draw, circle] (l1) {\strut$s_1$};

\node [right = 1 of k1] (sum) {\strut$\bigoplus_{k_1}$};

\node [right = 1 of sum, draw, circle] (k1') {\strut$k_1$};
\node [right = 6 of t1, draw, circle] (t2) {\strut$t_2$};
\node [below right = 1 of t2, draw, circle] (k2) {\strut$k_2$};
\node [below = 1 of k2, draw, circle] (l2) {\strut$s_2$};

\node [right = 2 of t2] (dots) {\strut$\cdots$};

\node [right = 2 of dots, draw, circle] (tn) {\strut$t_p$};
\node [below left = 1 of tn, draw, inner sep = .025cm, circle] (kn) {\strut$k_{p-1}$};
\node [below = 1 of kn, draw, inner sep = .025cm, circle] (ln) {\strut$s_{p-1}$};
\node [below right = 1 of tn, draw, circle] (j) {\strut$j$};

\draw [->] (t1) -- (i);
\draw [->] (t1) -- (k1);
\draw [->] (k1) -- (l1);

\draw [->] (t2) -- (k1');
\draw [->] (t2) -- (k2);
\draw [->] (k2) -- (l2);

\draw [->] (dots) -- (k2);
\draw [->] (dots) -- (kn);

\draw [->] (tn) -- (kn);
\draw [->] (tn) -- (j);
\draw [->] (kn) -- (ln);

\end{tikzpicture}
\caption{Iterative construction of $p$-trek zig-zag graph obtained via a 1-clique sum of a trek and $p-1$-trek zig-zag. }
\label{fig:ZigZagClique}
\end{subfigure}%

\begin{subfigure}{1\textwidth}
\centering%
\begin{tikzpicture}
\node [draw, circle] (i) {\strut$i$};
\node [above right = 1 of i, draw, circle, radius = .2cm] (t1) {\strut$t_1$};
\node [below right = 1 of t1, draw, circle] (k1) {\strut$k_1$};
\node [below = 1 of k1, draw, circle] (l1) {\strut$s_1$};

\node [right = 2 of t1, draw, circle] (t2) {\strut$t_2$};
\node [below right = 1 of t2, draw, circle] (k2) {\strut$k_2$};
\node [below = 1 of k2, draw, circle] (l2) {\strut$s_2$};

\node [right = 2 of t2] (dots) {\strut$\cdots$};

\node [right = 2 of dots, draw, circle] (tn) {\strut$t_p$};
\node [below left = 1 of tn, draw, inner sep = .025cm, circle] (kn) {\strut$k_{p-1}$};
\node [below = 1 of kn, draw, inner sep = .025cm, circle] (ln) {\strut$s_{p-1}$};
\node [below right = 1 of tn, draw, circle] (j) {\strut$j$};

\draw [->] (t1) -- (i);
\draw [->] (t1) -- (k1);
\draw [->] (k1) -- (l1);

\draw [->] (t2) -- (k1);
\draw [->] (t2) -- (k2);
\draw [->] (k2) -- (l2);

\draw [->] (dots) -- (k2);
\draw [->] (dots) -- (kn);

\draw [->] (tn) -- (kn);
\draw [->] (tn) -- (j);
\draw [->] (kn) -- (ln);

\end{tikzpicture}
\caption{The resulting $p$-trek zig-zag graph.}
\label{fig:ZigZagResult}
\end{subfigure}%

\caption{A zig-zag graph $\cz$ with endpoints $i$ and $j$ where each edge in the above graph represents a directed path. It is obtained by taking a 1-clique sum along the nodes $k_\alpha$ of the $p$ treks $T_\alpha$ with top node $t_\alpha$ and endpoints $s_{\alpha}$ and $s_{\alpha+1}$. The last step of this iterative process is pictured in (a) while the result is pictured in (b).}
\label{fig:ZigZagGraph}
\end{figure}
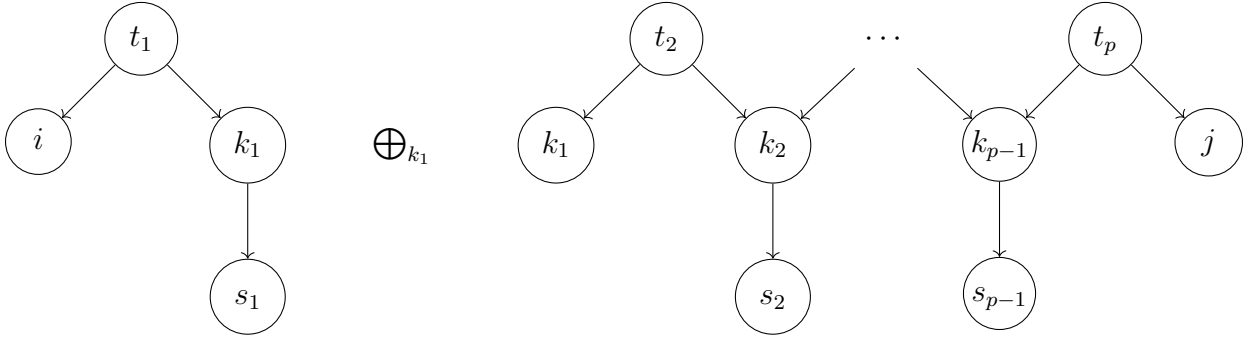
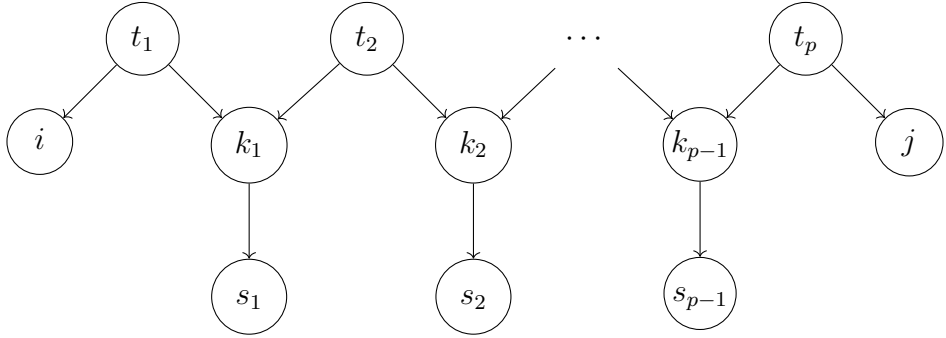

\begin{proposition}
\label{prop:NoCIsOnZigZag}
Let $\cz$ be a zig-zag graph with endpoints $i$ and $j$, and let $K \subseteq V(\cz) \setminus \{i, j\}$ be a set such that for each collider $k \in \cz$ it holds that $s \in K$ where $s$ is the unique sink in $\De(k)$. Then $i \not \indep j \mid K$ in the Lyapunov model associated to $\cz$. 
\end{proposition}

\begin{proof}
To see this we use \Cref{Prop:TrekDependence} inductively on the number $p$ of treks that $\cz$ is obtained from. If $\cz$ is simply a trek then the result already holds by \cref{Prop:TrekDependence}. So suppose that $\cz_p$ is a zig-zag graph with tails obtained from combining $p$ treks and that for any set $K' \subseteq V(\cz_p)$ of the required form it holds that $k \not \indep j \mid K'$ where $k$ and $j$ are the endpoints of $\cz_p$. Let $\ct$ be a trek from $i$ to $s$ containing a node labelled $k$ which is not on the same branch as $i$ and let $\cz_{p+1}$ be the 1-clique sum of $\ct$ and $\cz_{p}$ at $k$. Recall that this is exactly how every zig-zag graph is produced as pictured in \Cref{fig:ZigZagGraph}. Note that if $k$ and $i$ are on the same branch then checking that $i \not \indep j \mid K$ on $\cz_{p+1}$ naturally reduces to checking $i \not \indep j \mid K$ on a smaller zig-zag $\cz_p$ which immediately follows from the induction hypothesis.

To complete the induction we must show that $i \not \indep j \mid K$ where $K$ is of the form $K = T \cup L \cup K'$ such that $T \subseteq V(\ct) \setminus \{i\} \cup \De(k)$, $s \in L \subseteq \De(k)$ and $K' \subseteq \cz_p$ is a set such that $k \not \indep j \mid K'$. Lastly, let $l \in L$ be the node closest to $k$. Then observe that after rearranging the rows and columns of $\Sigma_{iK, jK}$ it has the block form
\[
\Sigma_{iK, jK} = 
\begin{blockarray}{cccc}
                   & jK'                              & l T                          & L \setminus l  \\
\begin{block}{l(ccc)}
l K'           & \Sigma_{lK', j K'}           & \ast                             & \ast  \\ 
i T                & 0                                & \Sigma_{iT, lT}              & \ast  \\
L \setminus l  & \Sigma_{L \setminus l, jK'}  & \Sigma_{L \setminus l, lT}  & \Sigma_{L \setminus l,L \setminus l}  \\
\end{block}
\end{blockarray}\;.
\]
Now observe that if we specialize this matrix by setting $m_{ab} = 0$ and $m_{la} = 0$ for all $a,b \in L \setminus l$ then the above block matrix immediately simplifies to
\[
\Sigma_{iK, jK} = 
\begin{blockarray}{cccc}
                   & jK'                              & l T                          & L \setminus l  \\
\begin{block}{l(ccc)}
l K'           & \Sigma_{lK', j K'}           & \ast                             & \ast  \\ 
i T                & 0                                & \Sigma_{iT, lT}              & \ast  \\
L \setminus l  & 0                                & 0                                & \Sigma_{L \setminus l, L \setminus l} \\
\end{block}
\end{blockarray}\;. 
\]
Thus we see that for any such matrix obtained from the submodel corresponding to this specialization, $|\Sigma_{iK, jK}| = |\Sigma_{lK', j K'}| \,|\Sigma_{iT, lT}| \,|\Sigma_{L \setminus l, L \setminus l}|$. Now observe that $|\Sigma_{lK', j K'}| \neq 0$ by our induction hypothesis since this corresponds to conditional dependence of a zig-zag with endpoints $l$ and $j$ which is composed of only $p$ treks. Further observe that $|\Sigma_{iT, lT}| \neq 0$ by \cref{cor:PerfectCorrelationDeterminant}. The last factor is a principal minor and thus is nonzero for any $\Sigma \in \cm_{\cz_{p+1}}$. So any $\Sigma$ which comes from the above submodel satisfies $|\Sigma_{iK, jK}| \neq 0$ and thus $i \not \indep j \mid K$ for the Lyapunov model $\cm_\cz$. 
\end{proof}

% \begin{proposition} \label{prop:NoCIsOnZigZag}
% \textcolor{benpurple}{more general version of the no-CI on zig-zag using perfect correlation. I think we can just state this here and then prove it in the appendix. }
% \end{proposition}

\begin{example}
Consider the zig-zag graph $\cz$ pictured in \Cref{fig:Trek with no CI}. Then $3 \not \indep 5 \mid \{1, 4\}$ since for any $\Sigma \in \cm_\cz$, it holds that
\[
\Sigma_{\{3,1,4\}, \{5, 1, 4\}} = 
\begin{blockarray}{cccc}
    & 5              & 4             & 1 \\
\begin{block}{l(ccc)}
4   & \sigma_{45}    & \sigma_{44}   & \sigma_{14}  \\ 
3   & 0              & \sigma_{34}   & \sigma_{13}  \\
1   & 0              & \sigma_{14}   & \sigma_{11}\\
\end{block}
\end{blockarray}\;. 
\]
As was shown in the previous proof, this matrix is block upper-triangular with blocks $\sigma_{45}$ and $\Sigma_{\{3,1\}, \{4,1\}}$. 
Thus $|\Sigma_{\{3,1,4\}, \{5, 1, 4\}}| = \sigma_{45} |\Sigma_{\{3,1\}, \{4,1\}}|$ where the second term in the product is guaranteed to be nonzero for almost all choices of parameters by \cref{Prop:BigTrekDependence}. 
\end{example}

We now have all the necessary tools needed to complete the proof of \cref{thm:CIImpliedByMarginals}. 

\begin{proof}[Proof of \cref{thm:CIImpliedByMarginals}]
If $V \setminus (\{i ,j \} \cup K)$ separates $i$ and $j$ in the marginal independence graph $\hatg$ then \cref{sec:ImpliedCI} immediately implies that $i \indep j \mid K$ holds for all $\Sigma \in \cm_\cg$. It remains to show that if $i \indep j \mid K$ then $V \setminus (\{i, j\} \cup K)$ separates $i$ in $j$ in $\hatg$. We instead show the contrapositive: if $i$ and $j$ are connected given $V \setminus (\{i, j\} \cup K)$ in $\hatg$, then there exists $\Sigma \in \cm_\cg$ such that $i \not \indep j \mid K$. We will prove this by constructing a specific $\Sigma$ using the trek rule described in \cref{prop:TrekRule} and then showing that $|\Sigma_{iK, jK}| \neq 0$ which by \Cref{cor:GaussianCI} implies that the conditional dependence $i \not\indep  j \mid K$ holds.

\begin{figure}
\newcommand{\scalefactor}{0.7}
\begin{subfigure}[t]{0.5\textwidth}
\centering%
\scalebox{\scalefactor}{%
\begin{tikzpicture}
\node [draw, circle] (t1) {\strut$t_1$};
\node [below left=1 and 0 of t1, draw, circle] (m) {\strut$m$};
\node [below left=1 and 0 of m, draw, circle] (i) {\strut$i$};
\node [right=1 of i, draw, circle] (l) {\strut$l$};
\node [right=1 of l, draw, circle] (t2) {\strut$t_2$};
\node [below=1 of t2, draw, circle] (j) {\strut$j$};
\node [right=1 of t2, draw, circle] (k) {\strut$k$};

\draw [->] (t1) -- (m);
\draw [->, blue] (m) -- (i);
\draw [->] (t1) -- (k);
\draw [->, blue] (t2) -- (m);
\draw [->] (m) -- (l);
\draw [->, blue] (t2) -- (j);
\end{tikzpicture}}
\caption{Left-left}
\label{fig:TrekIntersections:ll}
\end{subfigure}%
~
\begin{subfigure}[t]{0.5\textwidth}
\centering%
\scalebox{\scalefactor}{%
\begin{tikzpicture}
\node [draw, circle] (t1) {\strut$t_1$};
\node [below left=1 and 0 of t1, draw, circle] (m) {\strut$m$};
\node [below left=1 and 0 of m, draw, circle] (i) {\strut$i$};
\node [right=1 of i, draw, circle] (t2) {\strut$t_2$};
\node [below=1 of t2, draw, circle] (l) {\strut$l$};
\node [right=1 of t2, draw, circle] (j) {\strut$j$};
\node [right=1 of j, draw, circle] (k) {\strut$k$};

\draw [->] (t1) -- (m);
\draw [->, blue] (m) -- (i);
\draw [->] (t2) -- (l);
\draw [->] (t2) -- (m);
\draw [->] (t1) -- (k);
\draw [->, blue] (m) -- (j);
\end{tikzpicture}}
\caption{Left-right}
\label{fig:TrekIntersections:lr}
\end{subfigure}
\\[1em]
\begin{subfigure}[t]{0.5\textwidth}
\centering%
\scalebox{\scalefactor}{%
\begin{tikzpicture}
\node [draw, circle] (t1) {\strut$t_1$};
\node [below right=1 and 0 of t1, draw, circle] (m) {\strut$m$};
\node [below right=1 and 0 of m, draw, circle] (k) {\strut$k$};
\node [left=1 of k, draw, circle] (j) {\strut$j$};
\node [left=1 of j, draw, circle] (t2) {\strut$t_2$};
\node [left=1 of t2, draw, circle] (i) {\strut$i$};
\node [below=1 of t2, draw, circle] (l) {\strut$l$};

\draw [->, blue] (t1) -- (i);
\draw [->, blue] (t1) -- (m);
\draw [->] (m) -- (k);
\draw [->] (t2) -- (m);
\draw [->, blue] (m) -- (j);
\draw [->] (t2) -- (l);
\end{tikzpicture}}
\caption{Right-right}
\label{fig:TrekIntersections:rr}
\end{subfigure}%
~
\begin{subfigure}[t]{0.5\textwidth}
\centering%
\scalebox{\scalefactor}{%
\begin{tikzpicture}
\node [draw, circle] (i) {\strut$i$};
\node [right=1 of i, draw, circle] (l) {\strut$l$};
\node [right=1 of l, draw, circle] (k) {\strut$k$};
\node [right=1 of k, draw, circle] (j) {\strut$j$};
\node [above=1 of $(l.north)!0.5!(k.north)$, draw, circle] (m) {\strut$m$};
\node [above left=1 and 1 of m, draw, circle] (t1) {\strut$t_1$};
\node [above right=1 and 1 of m, draw, circle] (t2) {\strut$t_2$};

\draw [->, blue] (t1) -- (i);
\draw [->, blue] (t1) -- (m);
\draw [->] (m) -- (k);
\draw [->, blue] (t2) -- (j);
\draw [->, blue] (t2) -- (m);
\draw [->, blue] (m) -- (l);
\end{tikzpicture}}
\caption{Right-left}
\label{fig:TrekIntersections:rl}
\end{subfigure}%
\caption{The four possible intersections of a trek from $i$ over $t_1$ to $k$ and a trek from $l$ over $t_2$ to $j$ and how to shorten the walk from $i$ to $j$ (indicated in blue). Each arrow represents a directed path.}
\label{fig:TrekIntersections}
\end{figure}
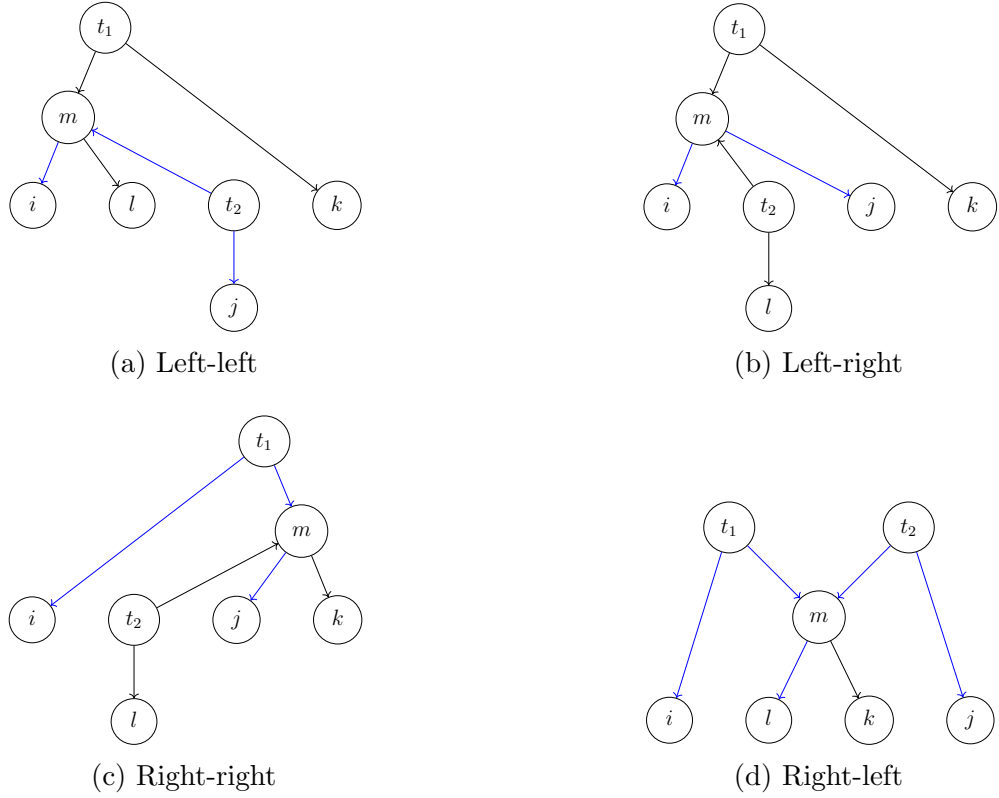

If $i$ and $j$ are connected in $\hatg$ given $V \setminus (\{i ,j \} \cup K)$ then there is path in $\hatg$ from $i$ to $j$ which uses only the vertices in $\{i, j\} \cup K$. Without loss of generality, this path is simple, i.e., it does not repeat a vertex. Let $i = k_0, k_1, \dots, k_s, k_{s+1} = j$ be the sequence of vertices on this path where $\{k_1, \dots, k_s\} \subseteq K$. Each edge $k_p \edge k_{p+1}$ in $\hatg$ corresponds to a trek in~$\cg$. Even though the path in $\hatg$ is simple, these treks can still intersect in~$\cg$. Suppose that the edges $k_p \edge k_{p+1}$ and $k_q \edge k_{q+1}$ with $p < q$ correspond to treks $T_1$ and $T_2$ in $\cg$ which intersect. There are four cases depending on where this intersection occurs which are depicted in \Cref{fig:TrekIntersections}. In the first three cases, there is a trek from $k_p$ to $k_{q+1}$ through the intersection vertex --- corresponding to a trek from $i$ to $j$ through $m$ in \Cref{fig:TrekIntersections}. This allows us to shorten the path in $\hatg$ by omitting $k_{p+q}, \dots, k_q$. The resulting path in $\hatg$ still has all its vertices contained in~$K$.
The fourth case \Cref{fig:TrekIntersections:rl} goes from $i$ over $m$ to $l$, back up to $m$ and then to~$j$, omitting $k$ from the path in~$\hatg$. The walk in $\cg$ traced out by this is a zig-zag with a tail $m \to l$ whose endpoint $l$ is still in~$K$.
For every pair of treks, multiple of these cases can occur simultaneously and each case can occur multiple times. One can freely choose any such intersection and resolve it according to \Cref{fig:TrekIntersections} which results in a shortened path in~$\hatg$. Iterating this procedure produces a simple path in $\hatg$ from $i$ to $j$ whose inner vertices are all contained in~$K$. The union of the corresponding treks in $\cg$ is a zig-zag with tails $\cz \subseteq \cg$ as in \Cref{fig:ZigZagGraph}. 

By construction the subgraph $\cz$ has the following property: for each collider $m$ on $\cz$, there is tail starting at $m$ and completely contained in $\cz$ whose end vertex is in~$K$. Hence, \Cref{prop:NoCIsOnZigZag} applies and shows that $i \not\indep j \mid K'$ on $\cm_\cz$ where $K'$ is the set of vertices on $\cz$ which are in $K$. Since $\cz$ is a subgraph of $\cg$, \Cref{lemma:Specialization} implies that $i \not\indep j \mid K$ on~$\cm_\cg$.
\end{proof}

\section*{Acknowledgements}

Tobias Boege was partially supported by the Academy of Finland grant number 323416 and by the Wallenberg Autonomous Systems and Software Program (WASP) funded by the Knut and Alice Wallenberg Foundation. Benjamin Hollering was supported by the Alexander von Humboldt Foundation. Mathias Drton, Pratik Misra, and Daniela Schkoda acknowledge support from the European Research Council (ERC) under the European Union’s Horizon 2020 research and innovation programme (grant agreement No 883818). Sarah Lumpp acknowledges support by the Munich Data Science Institute (MDSI) at Technical University of Munich as part of the focus topic CAUSE. Daniela Schkoda and Sarah Lumpp were also supported by the DAAD programme Konrad Zuse Schools of Excellence in Artificial Intelligence, sponsored by the Federal Ministry of Education and Research.

Part of this research was performed while some of the authors were visiting the Institute for Mathematical and Statistical Innovation (IMSI), which is supported by the National Science Foundation (Grant No. DMS-1929348). Benjamin Hollering and Tobias Boege also thank the Combinatorial Coworkspace (funded by Combinatorial Synergies SPP 2458). 

We thank Shiva Darshan for helpful discussions regarding the proof of \Cref{thm:MarginalIndependence}.

\bibliographystyle{plain}
\bibliography{refs.bib}

\clearpage

\appendix \section{Proofs}

\subsection{Proof of \cref{Prop:TrekDependence}}
\label{appendix:TrekDependence}
In this subsection we provide a proof of \cref{Prop:TrekDependence} by constructing a specific submodel of the Lyapunov model and then inductively showing that the determinant of the submatrix $\Sigma_{iK, jK}$ does not vanish for a trek $\ct$ with endpoints $i$ and $j$ where the conditioning set $K = V(\ct) \setminus \{i, j\}$. 

Let $\ct$ be a trek with endpoints $i$ and $j$ such that the length of the left branch, whose sink is $i$, is $\ell$ and the length of the right branch, whose sink is $j$, is $r$. Let $K = V(\ct) \setminus \{i, j\}$ be the set of all other nodes on $\ct$. Let $\cm'_\ct \subseteq \cm_\ct$ of the Lyapunov model on $\ct$ obtained from the trek rule in \cref{prop:TrekRule} by setting every edge parameter $m_{qp} = 1$ for all $p\rightarrow q \in E(\ct)$. In other words
\[
\cm'_\ct = \left \{ \Sigma \in \pd_n ~:~ \sigma_{ij} = \sum_{\substack{T : \text{$(\ell, r)$-trek} \\ \text{from $i$ to $j$}}} 2 \zeta^{\ell+r+1} \binom{\ell+r}{\ell} \right \}. 
\]
This follows directly from \cref{prop:TrekRule} since every trek monomial $m_T = \prod_{e \in T}m_e = 1$. Thus we can think of any matrix $\Sigma \in \cm'_\ct$ as a matrix whose entries are univariate polynomials in $\zeta$. Our goal is now to show that for almost all $\Sigma \in \cm'_\ct$, $|\Sigma_{iK, jK}| \neq 0$. To do this, we will focus on computing the lowest degree term in $|\Sigma_{iK, jK}|$ as a polynomial in $\zeta$.
Formally, the \emph{lowest degree term} of a univariate polynomial $f = \sum_{k=k_0}^n f_k \zeta^k$ with $f_{k_0} \not= 0$ is $\low(f) \defas f_{k_0} \zeta^{k_0}$.

Before we begin, note that every entry of every $\Sigma \in \cm_\ct'$ is divisible by $2\zeta$. In the following sequence of lemmas, we will instead consider a slight alteration of the above parameterization given by
\[
\cm''_\ct  = \left\{\Sigma \in \pd_n ~:~ \sigma_{ij} = \sum_{\substack{T : \text{$(\ell, r)$-trek} \\ \text{from $i$ to $j$}}} \zeta^{\ell+r} \binom{\ell+r}{\ell} \right\}.
\]
This simply corresponds to factoring out the common $2\zeta$ from each entry of $\Sigma$ which does not change whether or not a minor of $\Sigma$ vanishes. Thus if $\Sigma$ exhibits a conditional dependence then $2\zeta\,\Sigma$ will exhibit the same conditional dependence. To prove that our original model $\cm_\ct$ satisfies $i \not \indep j \mid K$, it suffices to prove it for a matrix in $\cm_\ct''$.

The following easy facts are used throughout the proofs in this section:
\begin{itemize}
\item The lowest degree term is multiplicative: $\low(fg) = \low(f) \low(g)$.
\item Since $\ct$ is a trek, there is a unique shortest trek between any two vertices and thus there is a well-defined notion of distance on the trek.
\item Let the shortest trek between $i$ and $j$ have length $(\ell, r)$. Since it is unique, we have $\low(\sigma_{ij}) = \binom{\ell+r}{\ell} \zeta^{\ell+r}$ for every $\Sigma = \Sigma(\zeta) \in \cm_\ct''$.
\item An important special case of this formula occurs when $i$ and $j$ are on the same branch of $\ct$. In that case $\ell = 0$ or $r = 0$ and the binomial coefficient is~$1$.
\end{itemize}
We begin with the following lemma concerning the lowest degree terms of principal minors.

\begin{lemma}
\label{lemma:LowestDegPrincMinor}
Let $\ct$ be a trek and $\Sigma \in \cm_\ct''$. For any $K \subseteq V$ we have $\low(|\Sigma_K|) = 1$. 
\end{lemma}

\begin{proof}
Under the parametrization of $\cm_\ct''$, any diagonal entry $\sigma_{ii}$ has $\low(\sigma_{ii}) = 1$ coming from the empty trek on $i$, while every off diagonal entry has $\deg(\low(\sigma_{ij})) \ge 1$. The result then follows immediately from the Leibniz formula for the determinant. 
\end{proof}

In the following, we compute minors of the form $|\Sigma_{iK,jK}|$ which are decisive for conditional independence by \Cref{cor:GaussianCI}. If $K = V \setminus \{i,j\}$, then this minor is, up to sign, the adjoint entry $\adj(\Sigma)_{ij}$. To facilitate inductive proofs via Schur's formula (which computes $|\Sigma_{iK,jK}|$ using entries of $\adj(\Sigma_K)$), it is important to not only show that these minors are nonzero, but also to consistently compute the signs.

\begin{convention} \label{conv:Ordering}
Throughout this section we assume the following numbering of the vertices $V = \{1,2,\ldots,n\}$ of the trek: the leaf on the left branch is labeled $1$, its parent $2$, and so on, until the top node is labeled $\ell+1$; from there we descend down with $\ell+2$ and so on until the leaf on the right branch is labeled $n = \ell+r+1$. Any minor $|\Sigma_{I,J}|$ is computed by ordering the rows and columns of the matrix with respect to this ordering.
\end{convention}

\begin{lemma}
\label{lemma:LowestDegAprMinor}
Let $\ct$ be a trek and $\Sigma \in \cm_\ct''$. Let $i, j \in V(\ct)$ be distinct and $(\ell, r)$ be the length of the shortest trek between them. For any $K \subseteq V(\ct) \setminus \{i,j\}$, either $|\Sigma_{iK,jK}| = 0$ or $\deg(\low(|\Sigma_{iK,jK}|)) \ge \ell+r$. If $i$ and $j$ are adjacent, then $\low(|\Sigma_{iK,jK}|) = \zeta$.
\end{lemma}

\begin{proof}
Consider the Leibniz formula
\[
  |\Sigma_{iK,jK}| = \sum_{\substack{\pi: iK \to jK \\ \text{bijective}}} \sgn(\pi) \prod_{k \in iK} \sigma_{k \pi(k)}.
\]
Note that $\Sigma_{iK,jK}$ is a matrix with a particular ordering on its rows and columns. This ordering fixes a bijection between $iK$ and $jK$ with respect to which $\sgn(\pi)$ is well-defined.

It suffices to show that each summand is either zero or has lowest term of $\text{degree} \ge \ell+r$. Fix a bijection $\pi: iK \to jK$. The image $\pi(i)$ either equals $j$ or is an element $k_1 \in K$. In~the latter case, we may apply $\pi$ again to get $\pi(k_1) = \pi^2(i)$ which either equals $j$ or $k_2 \in K$. Continuing in this way, collect all iterated images of $i$ into a set $A \subseteq iK$ including $i$ and excluding the $j$ which is eventually reached. Observe that
\begin{align}
  \label{eq:Leibniz}
  \deg \low\left(\prod_{k \in A} \sigma_{k \pi(k)}\right)
  = \sum_{k \in A} \deg \low(\sigma_{k \pi(k)}) \ge \low(\sigma_{ij}) = \ell + r,
\end{align}
since the sequence $\low(\sigma_{i \pi(i)}), \low(\sigma_{\pi(i)\pi^2(i)}), \dots, \low(\sigma_{\pi^m(i)j})$ given by $A$ encodes an undirected path from $i$ to $j$ which is at least as long as the shortest trek.

The lower bound in \eqref{eq:Leibniz} can only be attained with equality if $A$ is completely contained in the shortest trek $\ct'$ connecting $i$ and $j$ and if the successive images $\pi(i), \pi^2(i), \ldots$ are ordered such that $\pi(i)$ is the closest node to $i$, $\pi^2(i)$ the second closest and so on; while all elements of $K \setminus A$ are fixed points of~$\pi$ and do not contribute to the degree. In this case, the $\sigma_{kk}$, $k \in K \setminus A$, have lowest term~$1$ and the shortest treks between adjacent members of the sequence $i, \pi(i), \pi^2(i), \ldots, j$ give a disjoint decomposition of $\ct'$ hence achieving degree exactly~$\ell+r$ in total. If $i$ and $j$ are adjacent, then $\ell+r = 1$ and the only bijection $\pi$ which satisfies this sends $i$ to $j$ directly and fixes every other $k \in K$. Since $i$ and $j$ are adjacent, they appear in the same place in the ordering of $iK$ and $jK$, respectively, hence $\pi$ corresponds to the identity permutation of the rows and columns and thus we compute $\low(|\Sigma_{iK,jK}|) = \low(\sigma_{ij}) = \zeta$.
\end{proof}

As the previous proof shows, the lower degree bound $\ell+r$ is achievable and the Leibniz formula always has summands of this degree. Whether or not these summands cancel to cause a bump in lowest degree depends on the locations of $i$ and $j$ on the trek and to each other, as well as the conditioning set~$K$.

\begin{convention} \label{conv:Schur}
In the following proofs, we will use the Schur complement expansion of the determinant:
\begin{align}
  \label{eq:Schur}
  \sgn_K(i,j) |\Sigma_{iK,jK}| = \sigma_{ij} |\Sigma_K| - \Sigma_{i,K} \adj(\Sigma_K) \Sigma_{K,j} = \sigma_{ij} |\Sigma_K| - \sum_{p,q \in K} \sigma_{ip} \hsig_{pq} \sigma_{qj},
\end{align}
where $\hsig_{pq} = \adj(\Sigma_K)_{pq}$. To use Schur's formula to expand the determinant by the $(i,j)$-entry of $\Sigma_{iK,jK}$ as indicated, we have to permute the $i$\textsuperscript{th} row as well as the $j$\textsuperscript{th} column to be the first row and column, respectively. The symbol $\sgn_K(i,j)$ denotes the sign change in the determinant incurred by this. This sign depends only on the difference of the positions of $i$ in $iK$ and $j$ in $jK$. We will assume below that $K$ contains all interior nodes on the shortest trek $\ct'$ between $i$ and $j$. In this case and if $i \neq j$, $\sgn_K(i,j)$ can be computed as
\begin{align}
  \label{eq:Sign}
  \sgn_K(i,j) = (-1)^{n'} = (-1)^{|j-i|+1},
\end{align}
where $n'$ is the number of nodes on~$\ct'$, which equals $|j-i|+1$ under \Cref{conv:Ordering}.
\end{convention}

\begin{lemma}
\label{lemma:AdjointDegreeJump}
Let $\ct$ be a trek and $\Sigma \in \cm_\ct''$. Let $i, j \in V(\ct)$ be distinct, non-adjacent vertices on the same branch of~$\ct$. Denote the length of the shortest path $\ct'$ between $i$ and $j$ by $\ell > 1$. If $K \subseteq V(\ct) \setminus \{i,j\}$ contains all interior nodes on $\ct'$, then either
$|\Sigma_{iK,jK}| = 0$ or $\deg \low(|\Sigma_{iK,jK}|) > \ell$.
\end{lemma}

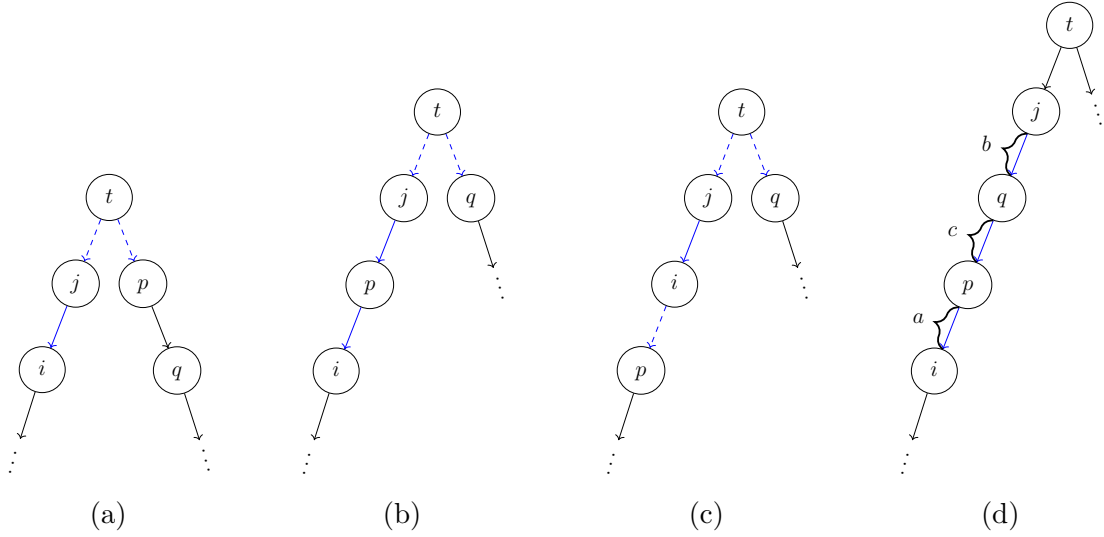
\begin{figure}
\newcommand{\scalefactor}{0.7}
\begin{subfigure}[t]{0.22\textwidth}
\centering%
\scalebox{\scalefactor}{\begin{tikzpicture}
\node [draw, circle] (1) {\strut$t$};
\node [draw, circle, below left=1 and 0 of 1] (j) {\strut$j$};
\node [draw, circle, below left=1 and 0 of j] (i) {\strut$i$};
\node [below left=1 and 0 of i] (d3) {\rotatebox{75}{\ldots}};
%\node [draw, circle, below left=1 and 0 of d3] (n1) {\vphantom{\strut$i$}};

\node [draw, circle, below right=1 and 0 of 1] (p) {\strut$p$};
\node [draw, circle, below right=1 and 0 of p] (q) {\strut$q$};
\node [below right=1 and 0 of q] (d6) {\rotatebox{-75}{\ldots}};
%\node [draw, circle, below right=1 and 0 of d6] (n2) {\vphantom{\strut$j$}};

\draw [->, dashed, blue]  (1) --  (j);
\draw [->, blue]  (j) --  (i);
\draw [->]  (i) -- (d3);
%\draw [->] (d3) -- (n1);

\draw [->, dashed, blue]  (1) -- (p);
\draw [->]  (p) -- (q);
\draw [->]  (q) -- (d6);
%\draw [->] (d6) -- (n2);
\end{tikzpicture}}
\caption{}%$i \to p$ is too long.}
\label{fig:AdjointZero:a}
\end{subfigure}%
~
\begin{subfigure}[t]{0.22\textwidth}
\centering%
\scalebox{\scalefactor}{\begin{tikzpicture}
\node [draw, circle] (1) {\strut$t$};
\node [draw, circle, below left=1 and 0 of 1] (j) {\strut$j$};
\node [draw, circle, below left=1 and 0 of j] (p) {\strut$p$};
\node [draw, circle, below left=1 and 0 of p] (i) {\strut$i$};
\node [below left=1 and 0 of i] (d4) {\rotatebox{75}{\ldots}};
%\node [draw, circle, below left=1 and 0 of d4] (n1) {\vphantom{\strut$i$}};

\node [draw, circle, below right=1 and 0 of 1] (q) {\strut$q$};
\node [below right=1 and 0 of q] (d6) {\rotatebox{-75}{\ldots}};
%\node [draw, circle, below right=1 and 0 of d6] (n2) {\vphantom{\strut$j$}};

\draw [->, dashed, blue]  (1) --  (j);
\draw [->, blue]  (j) --  (p);
\draw [->, blue]  (p) --  (i);
\draw [->]  (i) -- (d4);
%\draw [->] (d4) -- (n1);

\draw [->, dashed, blue]  (1) --  (q);
\draw [->]  (q) -- (d6);
%\draw [->] (d6) -- (n2);
\end{tikzpicture}}
\caption{}%$i \to p \to q$ is too long.}
\label{fig:AdjointZero:b}
\end{subfigure}
~
\begin{subfigure}[t]{0.22\textwidth}
\centering%
\scalebox{\scalefactor}{\begin{tikzpicture}
\node [draw, circle] (1) {\strut$t$};
\node [draw, circle, below left=1 and 0 of 1] (j) {\strut$j$};
\node [draw, circle, below left=1 and 0 of j] (i) {\strut$i$};
\node [draw, circle, below left=1 and 0 of i] (p) {\strut$p$};
\node [below left=1 and 0 of p] (d4) {\rotatebox{75}{\ldots}};
%\node [draw, circle, below left=1 and 0 of d4] (n1) {\vphantom{\strut$i$}};

\node [draw, circle, below right=1 and 0 of 1] (q) {\strut$q$};
\node [below right=1 and 0 of q] (d6) {\rotatebox{-75}{\ldots}};
%\node [draw, circle, below right=1 and 0 of d6] (n2) {\vphantom{\strut$j$}};

\draw [->, dashed, blue]  (1) --  (j);
\draw [->, blue]  (j) --  (i);
\draw [->, dashed, blue]  (i) --  (p);
\draw [->]  (p) -- (d4);
%\draw [->] (d4) -- (n1);

\draw [->, dashed, blue]  (1) --  (q);
\draw [->]  (q) -- (d6);
%\draw [->] (d6) -- (n2);
\end{tikzpicture}}
\caption{}%$p \to q$ is too long.}
\label{fig:AdjointZero:c}
\end{subfigure}~
\begin{subfigure}[t]{0.22\textwidth}
\centering%
\scalebox{\scalefactor}{\begin{tikzpicture}
\node [draw, circle] (1) {\strut$t$};
\node [draw, circle, below left=1 and 0 of 1] (j) {\strut$j$};
\node [draw, circle, below left=1 and 0 of j] (q) {\strut$q$};
\node [draw, circle, below left=1 and 0 of q] (p) {\strut$p$};
\node [draw, circle, below left=1 and 0 of p] (i) {\strut$i$};
\node [below left=1 and 0 of i] (d4) {\rotatebox{75}{\ldots}};
%\node [draw, circle, below left=1 and 0 of d4] (n1) {\vphantom{\strut$i$}};

\node [below right=1 and 0 of 1] (d6) {\rotatebox{-75}{\ldots}};
%\node [draw, circle, below right=1 and 0 of d6] (n2) {\vphantom{\strut$j$}};

\draw [->]  (1) --  (j);
\draw [->, blue]  (j) --  (q);
\draw [->, blue]  (q) --  (p);
\draw [->, blue]  (p) -- (i);
\draw [->]  (i) -- (d4);
%\draw [->] (d4) -- (n1);

\draw [->]  (1) -- (d6);
%\draw [->] (d6) -- (n2);

\draw[decoration={brace, amplitude=0.8em}, decorate, line width=1pt] (i) -- (p) node [midway, above, xshift=-1.5em, yshift=-0.5em] {\strut$a$};
\draw[decoration={brace, amplitude=0.8em}, decorate, line width=1pt] (p) -- (q) node [midway, above, xshift=-1.5em, yshift=-0.5em] {\strut$c$};
\draw[decoration={brace, amplitude=0.8em}, decorate, line width=1pt] (q) -- (j) node [midway, above, xshift=-1.5em, yshift=-0.58em] {\strut$b$};
\end{tikzpicture}}
\caption{}%$\ct'$ is perfectly covered.}
\label{fig:AdjointZero:d}
\end{subfigure}
\caption{The four cases in \Cref{lemma:AdjointDegreeJump}.}
\label{fig:AdjointZero}
\end{figure}

\begin{proof}
Without loss of generality, we may suppose by symmetry that both $i$ and $j$ are on the left branch of the trek and $i < j$. Then the shortest path connecting them is also the unique shortest trek connecting them, which has length~$(\ell, 0)$.

We perform induction on $\ell \ge 2$. The case $\ell = 2$ can be handled using the methods of \Cref{lemma:LowestDegAprMinor}. The shortest trek $\ct'$ between $i$ and $j$ is the form $i \to k \to j$ where $k \in K$ by assumption. Using the terminology of the proof of \Cref{lemma:LowestDegAprMinor}, the only two bijections $iK \to jK$ which contribute to the $\zeta^2$ term are:
\begin{enumerate}
\item $\pi_1$ sending $i \mapsto j$ and fixing every other element. Since $k$ is between $i$ and $j$ in the ordering, the element $i$ in $iK$ occupies the same place as $k$ in $jK$; and $k$ in $iK$ the same place as $j$ in $jK$. Hence $\pi_1$ corresponds to an adjacent transposition of the rows and columns and has $\sgn(\pi_1) = -1$. This summand is thus~$-\zeta^2$.
\item $\pi_2$ sending $i \mapsto k \mapsto j$ and fixing every other element. By similar reasoning as in the previous case, this corresponds to the identity permutation of the rows and columns and thus the contribution is $\zeta^2$.
\end{enumerate}
The two summands cancel and therefore either the entire determinant vanishes or the lowest term has $\text{degree} > 2$.

For the induction step, we use the Schur complement expansion \eqref{eq:Schur}. By~\cref{lemma:LowestDegPrincMinor} and the definition of $\cm_\ct''$, we know $\low(\sigma_{ij}|\Sigma_K|) = \zeta^{\ell}$. The goal is to show that for each choice of $p, q \in K$, either $\sigma_{ip} \hsig_{pq} \sigma_{qj} = 0$ or $\deg \low(\sigma_{ip} \hsig_{pq} \sigma_{qj}) \ge \ell$ and in case of equality to compute the coefficient of the $\zeta^\ell$ term.
We distinguish four cases (up to exchanging $p$ and~$q$ or switching branches) depending on the locations of $p$ and $q$ with respect to $i$ and $j$:
\begin{enumerate}[itemsep=0.6em, label=(\alph*)]
\item In case of \Cref{fig:AdjointZero:a}, notice that the shortest trek from $i$ to $p$ strictly includes the entire path from $i$ to $j$, so $\deg(\low(\sigma_{ip} \hsig_{pq} \sigma_{qj})) \ge \deg(\low(\sigma_{ip})) > \ell$.
\item In case of \Cref{fig:AdjointZero:b}, the reason is similar. \Cref{lemma:LowestDegAprMinor} shows $\deg(\low(\sigma_{ip} \hsig_{pq})) > \ell$.
\item In case of \Cref{fig:AdjointZero:c}, we know by \Cref{lemma:LowestDegAprMinor} that $\deg(\low(\hsig_{pq})) > \ell$.
\item\label{lemma:AdjointDegreeJump:4} The last case is depicted in \Cref{fig:AdjointZero:d}.
%The only remaining possibility is that $p$ and $q$ lie in between $i$ and $j$ which are on the same branch. Let $a$ be the distance from $i$ to $p$, let $b$ be the distance from $q$ to $j$, and $c$ be the distance from $p$ to $q$ in $\ct$ so $a + b + c = \ell$.
Then $\low(\sigma_{ip}) = \zeta^a$ and $\low(\sigma_{qj}) = \zeta^b$ and $a+b+c=\ell$. So
\[
\deg(\low(\sigma_{ip} \hsig_{pq} \sigma_{qj})) = \ell \iff \deg(\low(\hsig_{pq})) = c.
\]
Notice that $p$ and $q$ are on the same branch of~$\ct$. If they coincide ($c=0$), then $\low(\hsig_{pq}) = 1$ by \Cref{lemma:LowestDegPrincMinor}; if they are adjacent ($c=1$), \Cref{lemma:LowestDegAprMinor} yields $\hsig_{pq} = -|\Sigma_{K\setminus \{q\},K\setminus\{p\}}| = -\zeta$. Otherwise, the induction hypothesis ensures $\hsig_{pq} = 0$ or $\deg \low(\hsig_{pq}) > c$, which do not contribute to the lowest term. % Note that $K \setminus \{p,q\} still covers the shortest trek between $p$ and $q$ since $K$ covered the trek between $i$ and $j$. Thus, the induction hypothesis applies.
This implies that
\begin{align*}
\low\left(\sum_{p,q} \sigma_{ip} \hsig_{pq} \sigma_{qj}\right) &= \sum_{\substack{p, q \\ c = 0}} \left(\zeta^a \cdot 1 \cdot \zeta^b\right) + \sum_{\substack{p, q \\ c = 1}} \left(\zeta^a \cdot (-\zeta) \cdot \zeta^b\right) \\
&= \ell \zeta^\ell - (\ell - 1)\zeta^\ell = \zeta^\ell.
\end{align*}
\end{enumerate}
Thus the only case which contributes to the $\zeta^\ell$ term in $|\Sigma_{iK,jK}|$ is \ref{lemma:AdjointDegreeJump:4}. Its contribution cancels with $\low(\sigma_{ij}|\Sigma_K|) = \zeta^\ell$ in \eqref{eq:Schur}, implying the result. 
\end{proof}

\begin{proposition}
\label{Prop:BigTrekDependence}
Let $\ct$ be a trek and $i, j \in V(\ct)$ be nodes on different branches of~$\ct$ (in particular neither of them is the top node). Denote the length of the shortest path $\ct'$ between $i$ and $j$ by $(\ell, r)$. If $K \subseteq V(\ct) \setminus \{i,j\}$ contains all interior nodes on $\ct'$, then there exists $\Sigma \in \cm_\ct''$ such that $i \not \indep j \mid K$. 
\end{proposition}

\begin{proof}
Without loss of generality, we may assume that $i < j$ in the ordering, so that $i$ is on the left branch and $j$ on the right. We claim that $\low(|\Sigma_{iK,jK}|) = (-1)^{\ell+r+1} \binom{\ell+r-2}{\ell-1} \zeta^{\ell+r}$ which shows that $|\Sigma_{iK,jK}|$ is non-zero and yields conditional dependence for all but a finite number of values of $\zeta$.

We prove this statement by induction on $\ell, r \ge 1$, again using the Schur complement expansion \eqref{eq:Schur},
\[
  (-1)^{\ell+r+1} |\Sigma_{iK, jK}| = \sigma_{ij}|\Sigma_K|  - \Sigma_{i,K} \adj(\Sigma_K) \Sigma_{K,j} = \sigma_{ij}|\Sigma_K| - \sum_{p,q \in K} \sigma_{ip} \hsig_{pq} \sigma_{qj}
\]
with $\hsig_{pq} = \adj(\Sigma_K)_{pq}$.
By \cref{lemma:LowestDegPrincMinor} and the definition of $\cm_\ct''$, we have $\low(\sigma_{ij} |\Sigma_K|) = \binom{\ell+r}{\ell} \zeta^{\ell+r}$.
The other part of the expansion is a sum over $\sigma_{ip}\hsig_{pq}\sigma_{qj}$ which we deal with term by term, so fix $p, q \in K$. If~$p$ is strictly closer to $j$ than $q$ is (or, symmetrically, $q$ is closer to $i$ than $p$ is), then $\deg \low(\sigma_{ip} \sigma_{qj}) > \deg \low(\sigma_{ip} \sigma_{pj}) \ge \deg \low(\sigma_{ij}) = \ell+r$. Additionally, by \Cref{lemma:LowestDegAprMinor}, $p$ and $q$ must be contained in~$V(\ct')$, i.e. $i < p \le q < j$ in our ordering of vertices by \Cref{conv:Ordering}.
Hence, to find the lowest-order terms, we only need to consider
\begin{align}
  \label{eq:pqCond}
  \text{$p, q \in V(\ct') \setminus \{i,j\} \subseteq K$ such that $i < p \le q < j$.}
\end{align}

With these remarks, the basis of induction is easy to show. For $\ell = r = 1$ the only possible option in \eqref{eq:pqCond} is that $p = q = \text{top node of $\ct$}$. Then $\low(\sigma_{ip} \hsig_{pq} \sigma_{qj}) = \zeta \cdot 1 \cdot \zeta = \zeta^2$. With $\low(\sigma_{ij} |\Sigma_K|) = 2 \zeta^2$ we get $\low(|\Sigma_{iK,jK}|) = (-1)^3 \cdot (2\zeta^2 - \zeta^2) = -\zeta^2$, as required.

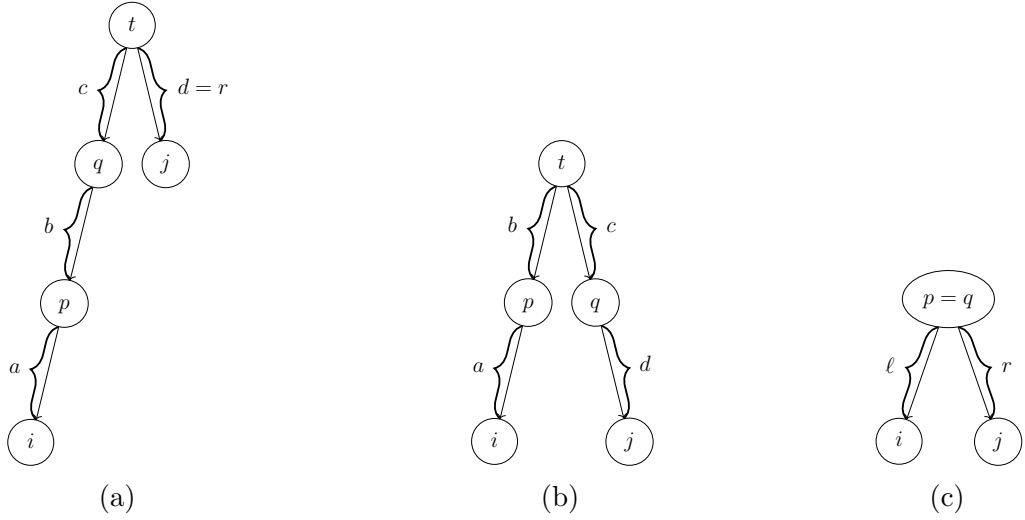
\begin{figure}
\newcommand{\scalefactor}{0.7}
\begin{subfigure}[t]{0.34\textwidth}
\centering%
\scalebox{\scalefactor}{\begin{tikzpicture}
\node [draw, circle] (1) {\strut$t$};
\node [draw, circle, below left=2 and 0 of 1] (q) {\strut$q$};
\node [draw, circle, below left=2 and 0 of q] (p) {\strut$p$};
\node [draw, circle, below left=2 and 0 of p] (i) {\strut$i$};;
\node [draw, circle, below right=2 and 0 of 1] (j) {\strut$j$};

\draw [->]  (1) -- (q);
\draw [->]  (q) -- (p);
\draw [->]  (p) -- (i);
\draw [->]  (1) -- (j);

\draw[decoration={brace, amplitude=0.8em}, decorate, line width=1pt] (q) -- (1) node [midway, above, xshift=-1.5em, yshift=-0.7em] {\strut$c$};
\draw[decoration={brace, amplitude=0.8em}, decorate, line width=1pt] (p) -- (q) node [midway, above, xshift=-1.5em, yshift=-0.7em] {\strut$b$};
\draw[decoration={brace, amplitude=0.8em}, decorate, line width=1pt] (i) -- (p) node [midway, above, xshift=-1.5em, yshift=-0.7em] {\strut$a$};
\draw[decoration={brace, amplitude=0.8em}, decorate, line width=1pt] (1) -- (j) node [midway, above, xshift=2.5em, yshift=-0.7em] {\strut$d=r$};
\end{tikzpicture}}
\caption{}%$p$ and $q$ on the same branch.}
\label{fig:Adjoint:a}
\end{subfigure}%
~
\begin{subfigure}[t]{0.34\textwidth}
\centering%
\scalebox{\scalefactor}{\begin{tikzpicture}
\node [draw, circle] (1) {\strut$t$};
\node [draw, circle, below left=2 and 0 of 1] (p) {\strut$p$};
\node [draw, circle, below left=2 and 0 of p] (i) {\strut$i$};

\draw [->]  (1) -- (p);
\draw [->]  (p) -- (i);

\node [draw, circle, below right=2 and 0 of 1] (q) {\strut$q$};
\node [draw, circle, below right=2 and 0 of q] (j) {\strut$j$};

\draw [->]  (1) -- (q);
\draw [->]  (q) -- (j);

\draw[decoration={brace, amplitude=0.8em}, decorate, line width=1pt] (p) -- (1) node [midway, above, xshift=-1.5em, yshift=-0.7em] {\strut$b$};
\draw[decoration={brace, amplitude=0.8em}, decorate, line width=1pt] (i) -- (p) node [midway, above, xshift=-1.5em, yshift=-0.7em] {\strut$a$};
\draw[decoration={brace, amplitude=0.8em}, decorate, line width=1pt] (1) -- (q) node [midway, above, xshift=1.5em, yshift=-0.7em] {\strut$c$}; 
\draw[decoration={brace, amplitude=0.8em}, decorate, line width=1pt] (q) -- (j) node [midway, above, xshift=1.5em, yshift=-0.7em] {\strut$d$};
\end{tikzpicture}}
\caption{}%$p$ and $q$ on different branches.}
\label{fig:Adjoint:b}
\end{subfigure}~
\begin{subfigure}[t]{0.25\textwidth}
\centering%
\scalebox{\scalefactor}{\begin{tikzpicture}
\node [draw, ellipse] (1) {\strut$p = q$};
\node [draw, circle, below left=2 and 0 of 1] (i) {\strut$i$};

\draw [->]  (1) -- (i);

\node [draw, circle, below right=2 and 0 of 1] (j) {\strut$j$};

\draw [->]  (1) -- (j);

\draw[decoration={brace, amplitude=0.8em}, decorate, line width=1pt] (i) -- (1) node [midway, above, xshift=-1.5em, yshift=-0.7em] {\strut$\ell$};
\draw[decoration={brace, amplitude=0.8em}, decorate, line width=1pt] (1) -- (j) node [midway, above, xshift=1.5em, yshift=-0.7em] {\strut$r$};
\end{tikzpicture}}
\caption{}%$p$ and $q$ are the top.}
\label{fig:Adjoint:c}
\end{subfigure}
\caption{The three cases in \Cref{Prop:BigTrekDependence}.}
\label{fig:Adjoint}
\end{figure}

In the general case, the reduction via \eqref{eq:pqCond} allows us to identify all relevant choices of $p,q$ with decompositions of $\ct'$ into three disjoint parts: the trek $i$ to $p$, the trek $p$ to $q$ and the trek $q$ to $j$.
There are three cases:

\begin{enumerate}[itemsep=0.6em, label=(\alph*)]
\item Suppose that $p$ and $q$ are both on the left branch of $\ct'$ and suppose that they are not both equal to the top node. Then the trek decomposes into four paths as in \Cref{fig:Adjoint:a}. We have $\low(\sigma_{ip} \sigma_{qj}) = \binom{c+r}{c} \zeta^{a+c+d}$. By \Cref{lemma:AdjointDegreeJump}, the lowest exponent of $\hat{\sigma}_{pq}$ in this case is $b$ if and only if $b = 0$ or $b = 1$. Otherwise it is larger than~$b$ and $a+b+c+d > \ell+r$ disqualifying those cases from contributing to the lowest terms of the overall expression, so assume $b = 0$ or $b = 1$. Then $\low(\hat{\sigma}_{pq}) = (-1)^{b} \zeta^b$ by \Cref{lemma:LowestDegAprMinor}. In total, this case contributes $\sum_{c = 1}^{\ell-1} \binom{c+r}{c}$ when $b=0$ and $-\sum_{c=0}^{\ell-2} \binom{c+r}{c}$ when $b=1$ to the lowest-order terms. (The different summation ranges for $c$ reflect that we must have $a \ge 1$ and $a+c = \ell$ and also that if $b = 0$, in order for $p = q$ not to be the top node, we must have $c \ge 1$ whereas if $b = 1$, we only need $c \ge 0$.) Together, this gives $\binom{\ell+r-1}{\ell-1}-1$. The case where $p$ and $q$ are both on the right branch is analogous and contributes $\binom{\ell+r-1}{r-1} = \binom{\ell+r-1}{\ell}-1$.

\item Suppose that $p$ is on the left branch and $q$ is on the right branch but neither of them is the top node. Then the trek decomposes into four paths as in \Cref{fig:Adjoint:b} and $\low(\sigma_{ip} \sigma_{qj}) = \zeta^{a+d}$. Here the induction hypothesis applies and computes $|\Sigma_{pL,qL}| = (-1)^{q-p+1} \binom{b+c-2}{b-1} \zeta^{b+c}$, where $L = K \setminus \{p,q\}$. (Note that $L$ contains all interior nodes on the shortest trek between $p$ and $q$ since $K$ had the same property for $i$ and $j$.) However, there is another sign to account for coming from the cofactor: $\hsig_{pq} = \adj(\Sigma_K)_{pq} = (-1)^{q-p} |\Sigma_{pL,qL}| = -\binom{b+c-2}{b-1} \zeta^{b+c}$. The sign only depends on $q-p$, again because $L$ contains all vertices between them in the ordering.
In total, this case contributes $-\sum_{b=1}^{\ell-1} \sum_{c=1}^{r-1} \binom{b+c-2}{b-1} = 1 - \binom{\ell+r-2}{\ell-1}$ to the lowest term.

\item There is a special case when $p$ and $q$ both coincide with the top node of the trek, as in \Cref{fig:Adjoint:c}. In this case, they are both simultaneously on the left and the right branch but this case was not counted in the previous bullet points. As in the base case of the induction, the contribution is $\low(\sigma_{ip} \hsig_{pq} \sigma_{qj}) = \zeta^{\ell+r}$.
\end{enumerate}
In total,
\begin{align*}
  &\hphantom{{}={}} (-1)^{\ell+r+1} \low(|\Sigma_{iK,jK}|) \\
  &= \left( \underbrace{\binom{\ell+r}{\ell} - \binom{\ell+r-1}{\ell-1} - \binom{\ell+r-1}{\ell}}_{= 0} + 2 
 - \left(1 - \binom{\ell+r-2}{\ell-1}\right) - 1 \right) \zeta^{\ell+r} \\
  &= \binom{\ell+r-2}{\ell-1} \zeta^{\ell+r},
\end{align*}
as claimed.
\end{proof}

The last remaining case needed to complete the proof of \cref{Prop:TrekDependence} is the case when $\ct$~is a trek such that $\ell = 0$ or $r = 0$, meaning that $\ct$ is actually a path. In this case \cref{lemma:AdjointDegreeJump} shows that the lowest degree term can have degree strictly higher than $\ell + r$. For this reason, we instead analyze the \emph{highest degree} term in $\zeta$ which appears in the expansion of~$|\Sigma_{iK, jK}|$. Let $\high(f)$ denote the highest degree non-zero term of a univariate polynomial~$f$.

\begin{lemma}
\label{lemma:PathDependence}
Let $\ct$ be a directed path from $i$ to $j$ and $K=V(\ct)\setminus \{i,j\}$ be the remaining vertices. Then there exists $\Sigma \in \cm_\ct''$ such that $i \not \indep j \mid K$. 
\end{lemma}
\begin{proof}
To show that there exist $\Sigma \in \cm_\ct''$ for which $i \not\indep j \mid K$, it is enough to show that $|\Sigma_{iK,jK}|$ as a univariate polynomial in $\zeta$ has a non-zero coefficient.
To this end, we compute the determinant of $\high(\Sigma_{iK,jK})$, which is the matrix obtained by taking the highest degree term of each entry of $\Sigma_{iK,jK}$, at $\zeta=1$. If this value is non-zero, it follows that $|\high(\Sigma_{iK,jK})| \neq 0$ is the highest degree term of $|\Sigma_{iK,jK}|$ which is thus non-zero as well.

We claim that $|\high(\Sigma_{iK,jK})|=1$ at $\zeta=1$.
Recall that $\cm_\ct''$ is defined by the parametrization which maps $\sigma_{ij}$ to the sum of $\zeta^{\ell+r}\binom{\ell+r}{\ell}$ over all possible treks between $i$ and~$j$. Hence, the entries of $\high(\Sigma_{iK,jK})$ only take into account the \emph{longest treks} between nodes (with the highest value of $\ell+r$).
For simplicity, we number the vertices from $1$ to $n$ with $i=1$ and $j=n$. Now, the longest and unique trek between $1$ and $k$ for any $k\leq n$ is precisely the directed path from $1$ to $k$. However, the longest trek between $a$ and $k$ for any $1<a<k$ is one where the left and right branches are directed paths from $1$ to $a$ and $1$ to $k$, respectively. This implies that the term with highest degree for $\sigma_{1k}$ has coefficient $1$ (as $\ell=0$) and for that of $\sigma_{ak}$ has coefficient $\binom{a+k-2}{a-1}$. This gives us that the $(a,b)$ entry of $\high(\Sigma_{\{1,2,3,\ldots, n-1\},\{2,3,\ldots,n\}})$ at $\zeta=1$ is precisely $\binom{a+b-2}{a-1}$. However, using the identity in \cite[Equation~(1)]{Pascal}, we can obtain an LU decomposition of this matrix where the diagonal terms of both L and U are~$1$. This implies that $|\high(\Sigma_{iK,jK})|=1$ at $\zeta=1$, which proves our claim. 
\end{proof}

\subsection{Proof Details for \cref{lemma:PerfectCorrelationMatrix}}
\label{appendix:PerfectCorrelation}

\begin{proof}[Proof Details for \cref{lemma:PerfectCorrelationMatrix}]
    In the proof of \cref{lemma:PerfectCorrelationMatrix}, we solve the following six equations resulting from the block-matrix partition for the corresponding blocks of $\Sigma^{(m)}$ for $m \to \infty$:
    \begin{itemize}
        \item[(A)] $M^*_{U} \Sigma^{(m)}_{U} + \Sigma^{(m)}_{U} \tp[-0]{M^*_{U}} + 2\Id_{k-1} = 0$ \\[2mm]
        This equation is uniquely solved by $\Sigma^{(m)}_{U} = \Sigma^*_{U}$ due to \ref{eq:M_star:a}.
        
        \item[(B)] $M^*_{U} \Sigma^{(m)}_{U,k} + m\Sigma^{(m)}_{U,k-1} -m \Sigma^{(m)}_{U,k} = 0$ \\[2mm]
        The matrix $M^*_{U} - m\Id_{k-1}$ is invertible due to the structure of $M^*$, so \cref{lemma:PerfectCorrelationLimitB} together with (A) yields
        \[
        \lim_{m \to \infty} \Sigma^{(m)}_{U,k} = \lim_{m \to \infty} \Sigma^{(m)}_{U,k-1} = \Sigma^*_{U,k-1}.
        \]
        
        \item[(C)] $M^*_{U} \Sigma^{(m)}_{U,W} + \Sigma^{(m)}_{U,W} \tp{M^*_{\psi(W)}} + \Sigma^{(m)}_{U,k}(m^*_{k,k-1}, 0, \dots, 0) = 0$ \\[2mm]
        Taking the limit of the left-hand side for $m \to \infty$ yields
        \begin{align*}
            & M^*_{U} \Sigma^{(m)}_{U,W} + \Sigma^{(m)}_{U,W} \tp{M^*_{\psi(W)}} + \Sigma^{(m)}_{U,k} (m^*_{k,k-1}, 0, \dots,0)  \\
            \xrightarrow[m \to \infty]{(B)} 
            & \ M^*_{U} \left( \lim_{m \to \infty} \Sigma^{(m)}_{U,W}\right) + \left( \lim_{m \to \infty} \Sigma^{(m)}_{U,W} \right) \tp{M^*_{\psi(W)}} + \Sigma^*_{U,k-1} (m^*_{k,k-1}, 0, \dots,0),
        \end{align*}
        where we employed (B) in the last summand. This last summand can further be rewritten as 
        \[
            \Sigma^*_{U,k-1} (m^*_{k,k-1}, 0,\dots,0) = \Sigma^*_{U} \tp{M^*_{\psi(W),U}}
        \]
        by attaching a $(k-2) \times (k-1)$ zero matrix to the existing row vector and thereby extending it to $\tp{M^*_{\psi(W),U}}$. Setting the limit to zero yields equation \ref{eq:M_star:b} which is solved by $\Sigma^*_{U,\psi(W)}$. This gives us $\lim_{m \to \infty}\Sigma^{(m)}_{U,W} = \Sigma^*_{U,\psi(W)}$.
        
        \item[(D)] $2-2m\Sigma^{(m)}_{k,k}+m\left(\Sigma^{(m)}_{k,k-1}+\Sigma^{(m)}_{k-1,k}\right) = 0$ \\[2mm]
        Solving for $\Sigma^{(m)}_{k,k}$, letting $m$ go to infinity, and using the result from (B) yields 
        \[
            \lim_{m \to \infty} \Sigma^{(m)}_{k,k} 
            = \lim_{m \to \infty} \left(\frac1m + \frac{\Sigma^{(m)}_{k,k-1}+\Sigma^{(m)}_{k-1,k}}{2} \right) 
            = \frac{\Sigma^*_{k-1,k-1}+\Sigma^*_{k-1,k-1}}{2} 
            = \Sigma^*_{k-1,k-1}.
        \]
        
        \item[(E)] $m \Sigma^{(m)}_{k-1,W} - m \Sigma^{(m)}_{k,W}+ \Sigma^{(m)}_{k,k} (m^*_{k,k-1}, 0,\dots,0) + \Sigma^{(m)}_{k,W} \tp{M^*_{\psi(W)}} = 0$ \\[2mm]
        The matrix $\tp{M^*_{\psi(W)}} - m\Id_{\nrnodes-k}$ is invertible due to the structure of $M^*$, so the transposed version of \cref{lemma:PerfectCorrelationLimitB} together with (C) and (D) yields
        \[
            \lim_{m \to \infty} \Sigma^{(m)}_{k,W} = \lim_{m \to \infty} \Sigma^{(m)}_{k-1,W} = \Sigma^*_{k-1,\psi(W)}.
        \]
        
        \item[(F)] $M^*_{\psi(W)} \Sigma^{(m)}_{W} + \Sigma^{(m)}_{W} \tp{M^*_{\psi(W)}} + 2\Id_{\nrnodes-k} \\
        + (m^*_{k,k-1}, 0,\dots,0)^T \Sigma^{(m)}_{k,W} + \Sigma^{(m)}_{W,k}(m^*_{k,k-1},0,\dots,0) = 0$ \\[2mm]
        We use that $\Sigma^{(m)}_{k,W} \to \Sigma^*_{k-1,\psi(W)}$ as well as $\Sigma^{(m)}_{W,k} \to \Sigma^*_{\psi(W),k-1}$ for $m \to \infty$ due to (E) and take the limit on both sides. This yields for the left-hand side
        \begin{align*}
            & M^*_{\psi(W)} \Sigma^{(m)}_{W} + \Sigma^{(m)}_{W} \tp{M^*_{\psi(W)}} + 2\Id_{\nrnodes-k}  \\
            & \qquad + (m^*_{k,k-1}, 0,\dots,0)^T \Sigma^{(m)}_{k,W} + \Sigma^{(m)}_{W,k}(m^*_{k,k-1},0,\dots,0) \\
            \xrightarrow[m \to \infty]{(E)} 
            & \ M^*_{\psi(W)} \left(\lim_{m \to \infty}\Sigma^{(m)}_{W}\right) + \left(\lim_{m \to \infty}\Sigma^{(m)}_{W}\right) \tp{M^*_{\psi(W)}} + 2\Id_{\nrnodes-k} \\
            & \qquad + (m^*_{k,k-1}, 0,\dots,0)^T \Sigma^*_{k-1,\psi(W)} + \Sigma^*_{\psi(W),k-1}(m^*_{k,k-1},0,\dots,0) \\
            = & \ \ M^*_{\psi(W)} \left(\lim_{m \to \infty}\Sigma^{(m)}_{W}\right) + \left(\lim_{m \to \infty}\Sigma^{(m)}_{W}\right) \tp{M^*_{\psi(W)}} + 2\Id_{\nrnodes-k} \\
            & \qquad + M^*_{\psi(W),U} \Sigma^*_{U,\psi(W)} + \Sigma^*_{\psi(W),U}\tp{M^*_{\psi(W),U}}, 
        \end{align*}
        where we employed the same technique of rewriting the last two summands as for (C). Setting the limit to zero yields equation \ref{eq:M_star:c} that is solved by $\Sigma^*_{\psi(W)}$. We conclude $\lim_{m \to \infty} \Sigma^{(m)}_{W} = \Sigma^*_{\psi(W)}$.
    \end{itemize} 
\end{proof}

\begin{lemma} \label{lemma:PerfectCorrelationLimitB}
    Let $A^{(m)}$ and $C^{(m)}$ be convergent sequences of $\nrnodes \times l$ matrices, and let $D,E,F \in \rr^{\nrnodes \times \nrnodes}$ be constant matrices such that $E - m \Id_\nrnodes$ is invertible for all $m \in \nn_{> 0}$.
    If a further matrix sequence $B^{(m)}$ satisfies that 
    \[
        m \cdot A^{(m)} - m \cdot B^{(m)} + D \cdot A^{(m)} + E \cdot B^{(m)} + F \cdot C^{(m)} =0
    \]
    for all $m \in \nn_{> 0}$, then
    \[
        \lim_{m\to \infty} B^{(m)} = \lim_{m\to\infty} A^{(m)}.
    \]
\end{lemma}

\begin{proof}
    By collecting the terms and dividing by $m$, we have 
    \[
       \left(\frac1m D + \Id_\nrnodes \right) A^{(m)} + \left(\frac1m E - \Id_\nrnodes\right) B^{(m)} +  \frac1m F \cdot C^{(m)} = 0.
    \]
    Since $\left(\frac1m E - \Id_\nrnodes\right)$ is invertible, we can write
    \[
        B^{(m)} = - \left(\frac1m E - \Id_\nrnodes\right)^{-1} \left(\left(\frac1m D + \Id_\nrnodes\right) A^{(m)} - \frac1m F \cdot C^{(m)}\right).
    \]
    As matrix inversion is continuous and $A^{(m)}$ and $C^{(m)}$ are convergent, letting $m$ go to infinity on both sides yields the result.
\end{proof}

\end{document}